\documentclass[11pt,a4paper]{article}

\usepackage{inputenc}
\usepackage{amsmath}
\usepackage{bm}
\usepackage{bbold}
\usepackage{amsthm}

\usepackage{hyperref}

\title{Solution of Linear Equations and Inequalities\\ in Idempotent Vector Spaces\thanks{International Journal of Applied Mathematics and Informatics, 2013. Vol.~7, no.~1, pp.~14-23.}
}

\author{Nikolai Krivulin\thanks{Faculty of Mathematics and Mechanics, St.~Petersburg State University, 28 Universitetsky Ave., St.~Petersburg, 198504, Russia, 
nkk@math.spbu.ru}
}

\date{}

\newtheorem{theorem}{Theorem}
\newtheorem{lemma}{Lemma}

\newtheorem{proposition}{Proposition}

\begin{document}

\maketitle

\begin{abstract}
Linear vector equations and inequalities are considered defined in terms of idempotent mathematics. To solve the equations, we apply an approach that is based on the analysis of distances between vectors in idempotent vector spaces. The approach reduces the solution of the equation to that of an optimization problem in the idempotent algebra setting. Based on the approach, existence and uniqueness conditions are established for the solution of equations, and a general solution to both linear equations and inequalities are given. Finally, a problem of simultaneous solution of equations and inequalities is also considered. 
\\

\textit{Key-Words:} idempotent vector space, linear equation, linear inequality, tropical optimization problem, existence condition
\end{abstract}

\section{Introduction}

Many applications of idempotent mathematics \cite{Vorobjev1963Theextremal,Vorobjev1967Extremal,Cuninghamegreen1979Minimax,Baccelli1993Synchronization,Kolokoltsov1997Idempotent,Litvinov1998Correspondence,Golan2003Semirings,Heidergott2006Maxplus,Butkovic2010Maxlinear} involve the solution of linear equations and inequalities defined on finite-dimensional semimodules over idempotent semifields (idempotent vector spaces). One of the problems that often arise is to solve an equation having the form
$$
A\bm{x}
=
\bm{d},
$$
where $A$ and $\bm{d}$ are given matrix and vector, $\bm{x}$ is an unknown vector, and multiplication is thought of in terms of idempotent algebra. Since the equation can be considered as representing linear dependence between vectors, the development of reasonable solution to the problem holds both practical and theoretical interest. Of particular importance are the methods that offer solutions in a compact vector form suitable for development of efficient computational algorithms and related software tools, including those intended for implementation in vector and parallel computers. 

Along with the above equation, an inequality
$$
A\bm{x}
\leq
\bm{d}
$$
that is considered component-wise constitutes another problem of interest in the idempotent algebra setting.

Among the early investigations of the problem of solving the equation and of its interplay with linear dependence of vectors are the works \cite{Vorobjev1963Theextremal,Vorobjev1967Extremal,Cuninghamegreen1979Minimax}. Further development of the question is given in many studies, including   \cite{Korbut1965Extremal,Korbut1972Extremal,Zimmermann1981Linear,Olsder1988Cramer,Cohen1989Algebraic,Baccelli1993Synchronization,Butkovic2010Maxlinear}.

To solve the equation where the matrix has no (idempotent) zero entries, an approach based on the concept of a covering set of rows for the matrix $A$ is proposed in \cite{Vorobjev1963Theextremal,Vorobjev1967Extremal}. With the approach, existence conditions are established and a procedure to find all solutions of the equation is described in terms of the covering sets. The maximum solution to the equation is given in the form $\bm{x}=A^{-}\otimes\bm{d}$, where $A^{-}$ is a pseudoinverse matrix  in the initial idempotent semimodule (called there extremal inverse matrix), and $\otimes$ denotes matrix-vector multiplication in a dual semimodule. In \cite{Korbut1965Extremal,Korbut1972Extremal}, the above approach is extended to investigate linear dependence in idempotent semimodules.
  
The development of the theory and methods in \cite{Cuninghamegreen1979Minimax} is aimed in particular at the solution of equations when the matrix $A$ may have zero entries. The operation of pseudoinversion is extended to such matrices (the matrix $A^{-}$ is called conjugate to $A$). For the solution, existence conditions in the form of an equality $A(A^{-}\otimes\bm{d})=\bm{d}$, where $\otimes$ is the multiplication in a dual semimodule, are given and uniqueness conditions are established. A procedure is proposed to determine the linear dependence between vectors. The results are further developed in \cite{Butkovic2010Maxlinear} to offer a combinatorial and an algebraic techniques for the solution of equations.

In \cite{Zimmermann1981Linear,Cohen1989Algebraic,Baccelli1993Synchronization}, a notion of a subsolution to the equation is introduced as any vector $\bm{x}$ that satisfies the condition $A\bm{x}\leq\bm{d}$. A residuation operation $\backslash$ is defined so that $A\backslash\bm{d}$ represents the maximal subsolution of the equation. It is shown that when an ordinary solution exists, it can be written in terms of a dual semimodule and then $A\backslash\bm{d}=A^{-}\otimes\bm{d}$. For an extended equation $A\bm{x}\oplus\bm{b}=\bm{d}$, where $\oplus$ denotes idempotent vector addition, a necessary and sufficient condition for the existence of its subsolutions is given in \cite{Cohen1989Algebraic,Baccelli1993Synchronization} in the form of an inequality $\bm{b}\leq\bm{d}$ that however suggests only necessary conditions for the actual solution.

Another approach that is based on the application of an idempotent analogue for the matrix determinant, known as dominant, is proposed in \cite{Olsder1988Cramer}. A solution technique is developed which uses Cramer's rule with the dominant in place of determinant. The implementation of the approach requires, however, that some sufficient constraints for both the matrix $A$ and the vector $\bm{d}$ to satisfy. 

In this paper another solution approach is described which uses the analysis of distances between vectors in idempotent vector spaces. As a metric, we take a distance function that involves only main binary operations of the semimodule supplemented with the operation of pseudoinversion. This allows to represent subsequent results in a compact vector form only in terms of the initial semimodule and give them clear and natural geometrical interpretation in the plane with the Cartesian coordinates. The results presented are based on implementation and further refinement of solutions that were first published in the papers \cite{Krivulin2005Onsolution,Krivulin2009Onsolution,Krivulin2009Methods} (see also \cite{Krivulin2012Asolution}) and were not fully available in English.

We start with a brief overview of preliminary algebraic definitions and results. Furthermore, the problem of solving the equation under study reduces to an optimization problem of finding the minimal distance from a vector to a linear span of vectors. We derive a comprehensive solutions to the optimization problem under quite general conditions. The obtained results are applied to give existence and uniqueness condition as well as to offer a general solution of the equation. Furthermore, a complete solution to the inequality is given. Finally, a problem of simultaneous solution of equations and inequalities is also considered.

\section{Preliminaries}

In this section, we present algebraic definitions, notations, and results based on \cite{Krivulin2009Onsolution,Krivulin2009Methods} to provide a background for subsequent analysis and solutions. Additional details and further results can be found in 
\cite{Vorobjev1967Extremal,Cuninghamegreen1979Minimax,Baccelli1993Synchronization,Kolokoltsov1997Idempotent,Litvinov1998Correspondence,Golan2003Semirings,Heidergott2006Maxplus,Butkovic2010Maxlinear}.

\subsection{Idempotent Semifield}

We consider a set $\mathbb{X}$ endowed with addition $\oplus$ and multiplication $\otimes$ and equipped with the zero $\mathbb{0}$ and the identity $\mathbb{1}$. The system $\langle\mathbb{X},\mathbb{0},\mathbb{1},\oplus,\otimes\rangle$ is assumed to be a linearly ordered radicable commutative semiring with idempotent addition and invertible multiplication, and it is commonly called idempotent semifield.

Idempotency of addition implies that $x\oplus x=x$ for all $x\in\mathbb{X}$. For any $x\in\mathbb{X}_{+}$, where $\mathbb{X}_{+}=\mathbb{X}\setminus\{\mathbb{0}\}$, there exists an inverse $x^{-1}$ such that $x^{-1}\otimes x=\mathbb{1}$. Furthermore, the power $x^{q}$ is defined for any $x\in\mathbb{X}_{+}$ and a rational $q$. Specifically, for any integer $p\geq0$, we have
$$
x^{0}
=\mathbb{1},
\qquad
x^{p}
=
x^{p-1}x,
\qquad
x^{-p}
=
(x^{-1})^{p}.
$$

In what follows, we drop the multiplication sign $\otimes$ and use the power notation only in the above sense.

The linear order defined on $\mathbb{X}$ is assumed to be consistent with a partial order that is induced by idempotent addition to involve that $x\leq y$ if and only if $x\oplus y=y$. From the last definition it follows that addition possesses an extremal property in the form of inequalities
$$
x\leq x\oplus y,
\qquad
y\leq x\oplus y,
$$
and that both addition and multiplication are isotonic.

Below, the relation symbols and the operator $\min$ are thought in terms of the above linear order. 

Note that we have $x\geq\mathbb{0}$ for all $x\in\mathbb{X}$. We also assume that the set $\mathbb{X}$ includes (or can be extended by) a maximal element $\infty$ such that $x\leq\infty$ for all $x\in\mathbb{X}$.

Examples of the linearly ordered radicable idempotent semifield under consideration include 
\begin{align*}
\mathbb{R}_{\max,+}
&=
\langle\mathbb{R}\cup\{-\infty\},-\infty,0,\max,+\rangle,
\\
\mathbb{R}_{\min,+}
&=
\langle\mathbb{R}\cup\{+\infty\},+\infty,0,\min,+\rangle,
\\
\mathbb{R}_{\max,\times}
&=
\langle\mathbb{R}_{+}\cup\{0\},0,1,\max,\times\rangle,
\\
\mathbb{R}_{\min,\times}
&=
\langle\mathbb{R}_{+}\cup\{+\infty\},+\infty,1,\min,\times\rangle,
\end{align*}
where $\mathbb{R}$ is the set of real numbers, $\mathbb{R}_{+}=\{x\in\mathbb{R}|x>0\}$.

Specifically, the semifield $\mathbb{R}_{\max,+}$ has its null and identity defined as $\mathbb{0}=-\infty$ and $\mathbb{1}=0$. For each $x\in\mathbb{R}$, there exists an inverse $x^{-1}$ equal to $-x$ in conventional arithmetic. For any $x,y\in\mathbb{R}$, the power $x^{y}$ corresponds to the arithmetic product $xy$. The order induced by the idempotent addition coincides with the natural linear order on $\mathbb{R}$. The maximal element is given by $+\infty$.

In $\mathbb{R}_{\min,\times}$, we have $\mathbb{0}=+\infty$ and $\mathbb{1}=1$. The inverse and power notations have the same interpretation as in the conventional algebra. The relation $\leq$ defines a reverse order to the linear order on $\mathbb{R}$. The role of the maximal element is performed by $0$. 

The semifields $\mathbb{R}_{\max,+}$, $\mathbb{R}_{\min,+}$, $\mathbb{R}_{\max,\times}$, and $\mathbb{R}_{\min,\times}$ are isomorphic to each other. Fig.~\ref{F-SRI} offers a diagram that represents isomorphism maps for these semifields. 
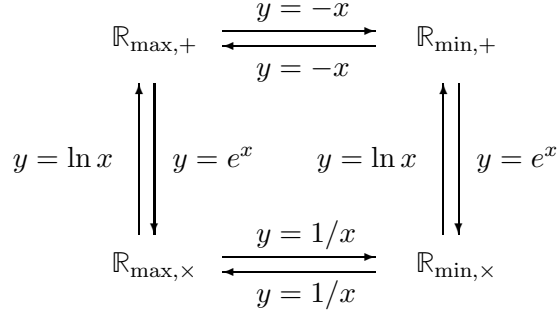
\begin{figure}[ht]
\setlength{\unitlength}{1mm}
\begin{center}

\begin{picture}(70,45)

\put(13,35){$\mathbb{R}_{\max,+}$}

\put(32,39){$y=-x$}
\put(28,37){\vector(1,0){20}}

\put(48,35){\vector(-1,0){20}}
\put(32,31){$y=-x$}

\put(53,35){$\mathbb{R}_{\min,+}$}

\put(0,19){$y=\ln x$}
\put(17,10){\vector(0,1){20}}
\put(19,30){\vector(0,-1){20}}
\put(21,19){$y=e^{x}$}

\put(40,19){$y=\ln x$}
\put(57,10){\vector(0,1){20}}
\put(59,30){\vector(0,-1){20}}
\put(61,19){$y=e^{x}$}

\put(13,5){$\mathbb{R}_{\max,\times}$}

\put(32,9){$y=1/x$}
\put(28,7){\vector(1,0){20}}
\put(48,5){\vector(-1,0){20}}
\put(32,1){$y=1/x$}

\put(53,5){$\mathbb{R}_{\min,\times}$}

\end{picture}
\caption{Isomorphism of $\mathbb{R}_{\max,+}$, $\mathbb{R}_{\min,+}$, $\mathbb{R}_{\max,\times}$, and $\mathbb{R}_{\min,\times}$.}\label{F-SRI}
\end{center}

\end{figure}

As an example of an idempotent semiring that is not a simifield, one can consider
$$
\mathbb{R}_{\max,\min}
=
\langle\mathbb{R}\cup\{-\infty,+\infty\},-\infty,+\infty,\max,\min\rangle.
$$
In this semiring, it holds that $\mathbb{0}=-\infty$ and $\mathbb{1}=+\infty$. Inverse elements with respect to the multiplication $\otimes$ defined to be $\min$ do not exist, whereas the power notation is undefined. The order induced by the addition $\oplus$ corresponds to the natural linear order. The maximal element is $+\infty$.

Now we introduce a distance function $\rho$ on $\mathbb{X}$ as follows. For any $x,y\in\mathbb{X}_{+}$, we define
$$
\rho(x,y)
=
y^{-1}x\oplus x^{-1}y.
$$

Since the function $\rho$ takes values in the segment $[\mathbb{1},\infty)$, it is natural to put $\rho(x,y)=\mathbb{1}$ when $x=y=\mathbb{0}$. We also assume that $\rho(x,y)=\infty$ if one of the arguments $x$ and $y$ is zero while the other is not.

In the semifield $\mathbb{R}_{\max,+}$ for all $x,y\in\mathbb{R}$, the function $\rho$ coincides with the ordinary distance $d(x,y)=|x-y|$. Due to the isomorphism  between semifields, the function $\rho$ induces a distance function in each semifield $\mathbb{R}_{\max,\times}$, $\mathbb{R}_{\min,+}$, and $\mathbb{R}_{\min,\times}$. Specifically, in $\mathbb{R}_{\max,\times}$, we have
$$
\rho^{\prime}(x,y)
=
\ln(y^{-1}x\oplus x^{-1}y).
$$

The function $\rho$ possesses the symmetry property and satisfies the triangle inequality in all semifields $\mathbb{R}_{\max,+}$, $\mathbb{R}_{\max,\times}$, $\mathbb{R}_{\min,+}$, and $\mathbb{R}_{\min,\times}$. Moreover, for every such semifield, the function can always be converted into an actual metric by scaling its value by an appropriate isomorphism. Below, for simplicity, we leave out the isomorphism maps, and take the function $\rho$ as a metric for all semifields.

\subsection{Idempotent Vector Space}

Consider the Cartesian power $\mathbb{X}^{m}$ with column vectors as its elements. A vector with all elements equal to $\mathbb{0}$ is the zero vector. A vector is regular if it has no zero elements. 

For any two vectors $\bm{a}=(a_{i})$ and $\bm{b}=(b_{i})$ in $\mathbb{X}^{m}$, and a scalar $x\in\mathbb{X}$, addition and scalar multiplication are defined component-wise as follows
$$
\{\bm{a}\oplus\bm{b}\}_{i}
=
a_{i}\oplus b_{i},
\qquad
\{x\bm{a}\}_{i}
=
xa_{i}.
$$

Endowed with these operations, the set $\mathbb{X}^{m}$ forms a semimodule over the idempotent semifield $\mathbb{X}$ and it is referred to as the idempotent vector space.

Fig.~\ref{F-VA} gives geometrical illustrations of the addition in the space $\mathbb{R}_{\max,+}^{2}$ with the Cartesian coordinates on the plane. Note that the left  example can equally refer to $\mathbb{R}_{\max,\times}^{2}$.
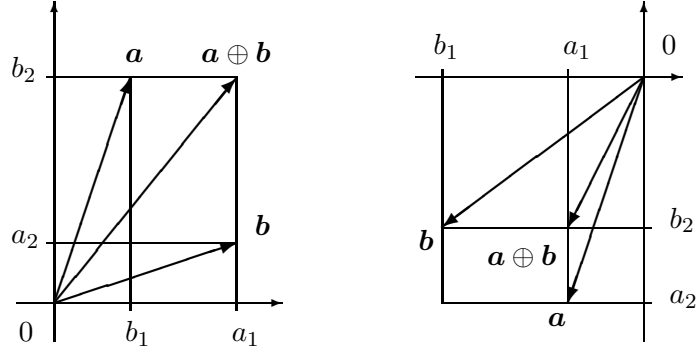
\begin{figure}[ht]
\setlength{\unitlength}{1mm}
\begin{center}
\begin{picture}(35,45)

\put(0,5){\vector(1,0){35}}
\put(5,0){\vector(0,1){45}}

\put(5,5){\thicklines\vector(1,3){10}}
\put(15,35){\line(0,-1){31}}

\put(5,5){\thicklines\vector(3,1){24}}
\put(29,13){\line(-1,0){25}}

\put(5,5){\thicklines\line(4,5){24}}
\put(26,31.75){\thicklines\vector(1,1){3}}

\put(29,35){\line(-1,0){25}}

\put(29,35){\line(0,-1){31}}

\put(0,0){$0$}

\put(14,0){$b_{1}$}
\put(28,0){$a_{1}$}

\put(-1,13){$a_{2}$}
\put(-1,35){$b_{2}$}

\put(14,37){$\bm{a}$}

\put(31,14){$\bm{b}$}

\put(24,37){$\bm{a}\oplus\bm{b}$}

\end{picture}
\hspace{15\unitlength}
\begin{picture}(35,45)

\put(0,35){\vector(1,0){35}}
\put(30,0){\vector(0,1){45}}

\put(30,35){\thicklines\vector(-1,-3){10}}
\put(3.5,5){\line(1,0){27.5}}
\put(20,5){\line(0,1){31}}

\put(30,35){\thicklines\vector(-4,-3){26.5}}
\put(3.5,15){\line(1,0){27.5}}
\put(3.5,5){\line(0,1){31}}

\put(30,35){\thicklines\vector(-1,-2){10}}

\put(32,38){$0$}

\put(19,38){$a_{1}$}
\put(33,5){$a_{2}$}

\put(2,38){$b_{1}$}
\put(33,15){$b_{2}$}

\put(0,12){$\bm{b}$}

\put(17,2){$\bm{a}$}

\put(9,10){$\bm{a}\oplus\bm{b}$}

\end{picture}
\end{center}
\caption{Vector addition in $\mathbb{R}_{\max,+}^{2}$.}\label{F-VA}
\end{figure}

Idempotent addition of two vectors in $\mathbb{R}_{\max,+}^{2}$ follows the ``rectangle rule'' that defines the sum as the upper right vertex of a rectangle formed by  lines drawn through the end points of the vectors, parallel to the coordinate axes.

Scalar multiplication of vectors in $\mathbb{R}_{\max,+}^{2}$ is equivalent to shifting the end point of the vector in the direction at $45^{\circ}$ to the axes (see Fig.~\ref{F-SM}, left), whereas in $\mathbb{R}_{\max,\times}^{2}$ it has conventional geometrical interpretation (Fig.~\ref{F-SM}, right).
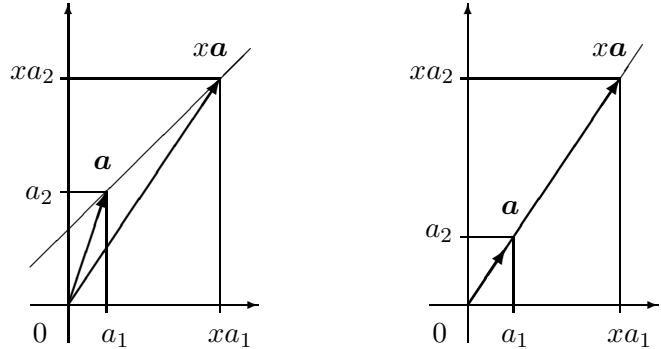
\begin{figure}[ht]
\setlength{\unitlength}{1mm}
\begin{center}

\begin{picture}(35,45)

\put(5,5){\vector(1,0){30}}
\put(10,0){\vector(0,1){45}}

\put(10,5){\thicklines\vector(1,3){5}}
\put(15,20){\line(0,-1){16}}
\put(15,20){\line(-1,0){6}}

\put(10,5){\thicklines\vector(2,3){20}}
\put(30,35){\line(0,-1){31}}
\put(30,35){\line(-1,0){21}}

\put(5,10){\line(1,1){29}}

\put(5,0){$0$}
\put(13,23){$\bm{a}$}
\put(26,38){$x\bm{a}$}

\put(4,19){$a_{2}$}
\put(2,35){$xa_{2}$}

\put(14,0){$a_{1}$}
\put(28,0){$xa_{1}$}

\end{picture}
\hspace{15\unitlength}
\begin{picture}(35,45)

\put(5,5){\vector(1,0){30}}
\put(10,0){\vector(0,1){45}}

\put(10,5){\thicklines\vector(2,3){5}}
\put(16,14){\line(0,-1){10}}
\put(16,14){\line(-1,0){7}}

\put(10,5){\thicklines\vector(2,3){20}}
\put(30,35){\line(0,-1){31}}
\put(30,35){\line(-1,0){21}}

\put(10,5){\line(2,3){23}}

\put(5,0){$0$}
\put(14,17){$\bm{a}$}
\put(26,38){$x\bm{a}$}

\put(4,14){$a_{2}$}
\put(2,35){$xa_{2}$}

\put(14,0){$a_{1}$}
\put(28,0){$xa_{1}$}

\end{picture}

\end{center}
\caption{Scalar multiplication in $\mathbb{R}_{\max,+}^{2}$ (left) and in $\mathbb{R}_{\max,\times}^{2}$ (right).}\label{F-SM}
\end{figure}

For any vectors $\bm{a},\bm{b}\in\mathbb{X}^{m}$, the extremal property of addition leads to component-wise vector inequalities
$$
\bm{a}
\leq
\bm{a}\oplus\bm{b},
\qquad
\bm{b}
\leq
\bm{a}\oplus\bm{b}.
$$ 

Furthermore, vector and scalar inequalities $\bm{a}\leq\bm{b}$ and $x\leq y$ imply that the inequalities 
$$
\bm{a}\oplus\bm{c}
\leq
\bm{b}\oplus\bm{c},
\qquad
x\bm{a}
\leq
x\bm{b},
\qquad
x\bm{c}
\leq
y\bm{c}
$$
are valid for all $\bm{c}\in\mathbb{X}^{m}$, and so both vector addition and scalar multiplication are isotone in each argument.

\subsection{Matrix Algebra}

Consider matrices having entries in $\mathbb{X}$. For conforming matrices $A=(a_{ij})$, $B=(b_{ij})$, and $C=(c_{ij})$, matrix addition and multiplication together with multiplication by a scalar $x\in\mathbb{X}$ follow the conventional rules 
\begin{gather*}
\{A\oplus B\}_{ij}
=
a_{ij}\oplus b_{ij},
\qquad
\{B C\}_{ij}
=
\bigoplus_{k}b_{ik}c_{kj},
\\
\{xA\}_{ij}=xa_{ij}.
\end{gather*}

Specifically, a matrix $A=(a_{ij})\in\mathbb{X}^{m\times n}$ is multiplied by a vector $\bm{x}=(x_{i})\in\mathbb{X}^{n}$ to result in a vector with elements
$$
\{A\bm{x}\}_{i}
=
a_{i1}x_{1}\oplus\cdots\oplus a_{in}x_{n}.
$$

The matrix operations is component-wise isotone in each argument.

A matrix with all entries equal to zero is called the zero matrix and denoted by $\mathbb{0}$.

A matrix is row regular (column regular) if it has no zero rows (columns).

A square matrix is diagonal if its off-diagonal entries are zero. The diagonal matrix $I=\mathop\mathrm{diag}(\mathbb{1},\ldots,\mathbb{1})$ is the identity.

For any nonzero column vector $\bm{x}=(x_{i})\in\mathbb{X}^{n}$, we define a row vector $\bm{x}^{-}=(x_{i}^{-})$, where $x_{i}^{-}=x_{i}^{-1}$ if $x_{i}\ne\mathbb{0}$, and $x_{i}^{-}=\mathbb{0}$ otherwise.

If $\bm{x}$ is a nonzero vector, then $\bm{x}^{-}\bm{x}=\mathbb{1}$.

Suppose that both $\bm{x}$ and $\bm{y}$ are regular vectors. The component-wise inequality $\bm{x}\leq\bm{y}$ implies $\bm{x}^{-}\geq\bm{y}^{-}$. Furthermore, it is not difficult to verify the inequality
\begin{equation}
\bm{x}\bm{y}^{-}\geq(\bm{x}^{-}\bm{y})^{-1}I.
\label{I-xyxyI}
\end{equation}

Indeed, since $\bm{x}^{-}\bm{y}=x_{1}^{-1}y_{1}\oplus\cdots\oplus x_{n}^{-1}y_{n}\geq x_{i}^{-1}y_{i}$, we have $x_{i}y_{i}^{-1}\geq(\bm{x}^{-}\bm{y})^{-1}$ for all $i=1,\ldots,n$, and so
$$
\bm{x}\bm{y}^{-}
\geq
\mathop{\mathrm{diag}}(x_{1}y_{1}^{-1},\ldots,x_{n}y_{n}^{-1})
\geq
(\bm{x}^{-}\bm{y})^{-1}I.
$$

When $\bm{y}=\bm{x}$ the inequality \eqref{I-xyxyI} takes the form $\bm{x}\bm{x}^{-}\geq I$. By applying \eqref{I-xyxyI} in this form, one can get the inequality


\begin{equation}
(A\bm{x})^{-}A
\leq
\bm{x}^{-},
\label{I-AxAx}
\end{equation}
which is valid for any nonzero matrix $A\in\mathbb{X}^{m\times n}$. In fact, if $\bm{x}\in\mathbb{X}_{+}^{n}$, then we have $(A\bm{x})^{-}A\leq(A\bm{x})^{-}A\bm{x}\bm{x}^{-}=\bm{x}^{-} $.

\subsection{Linear Dependence}

Consider a system of vectors $\bm{a}_{1},\ldots,\bm{a}_{n}\in\mathbb{X}^{m}$. As usual, a vector $\bm{b}\in\mathbb{X}^{m}$ is linearly dependent on the system if it admits representation as a linear combination
$$
\bm{b}
=
x_{1}\bm{a}_{1}\oplus\cdots\oplus x_{n}\bm{a}_{n}
$$
with coefficients $x_{1},\ldots,x_{n}\in\mathbb{X}$.

The linear span of $\bm{a}_{1},\ldots,\bm{a}_{n}$ is defined as the set of all linear combinations that form an idempotent subspace
$$
\mathop\mathrm{span}(\bm{a}_{1},\ldots,\bm{a}_{m})
=
\left\{\left.\bigoplus_{i=1}^{m}x_{i}\bm{a}_{i}\right|x_{1},\ldots,x_{m}\in\mathbb{X}\right\}.
$$

Geometrical examples of linear spans are given in Fig.~\ref{F-LC}.
\begin{figure}[ht]
\setlength{\unitlength}{1mm}
\begin{center}

\begin{picture}(35,45)

\put(0,5){\vector(1,0){35}}
\put(5,0){\vector(0,1){45}}

\put(5,5){\thicklines\vector(1,3){8}}
\put(5,5){\thicklines\vector(1,2){16}}

\put(5,5){\thicklines\vector(4,1){16}}
\put(5,5){\thicklines\vector(2,1){24}}

\put(0,16){\thicklines\line(1,1){24}}
\multiput(1,17)(1,1){23}{\line(1,0){1}}

\put(12,0){\thicklines\line(1,1){22}}
\multiput(13,1)(1,1){21}{\line(-1,0){1}}

\put(29,37){\line(-1,0){25}}
\put(29,37){\line(0,-1){33}}

\put(13,29){\line(-1,0){9}}
\put(21,9){\line(0,-1){5}}

\put(5,5){\thicklines\vector(3,4){24}}

\put(0,0){$0$}

\put(9,31){$\bm{a}_{2}$}
\put(0,33){$x_{2}$}
\put(13,39){$x_{2}\bm{a}_{2}$}

\put(23,7){$\bm{a}_{1}$}
\put(24,0){$x_{1}$}
\put(31,15){$x_{1}\bm{a}_{1}$}


\end{picture}
\hspace{15\unitlength}
\begin{picture}(35,45)

\put(0,5){\vector(1,0){35}}
\put(5,0){\vector(0,1){45}}

\put(5,5){\thicklines\vector(1,3){7}}
\put(5,5){\thicklines\vector(1,3){10}}

\put(5,5){\thicklines\line(1,3){12}}
\multiput(5,5)(0.5,1.5){24}{\line(1,0){1}}

\put(5,5){\thicklines\vector(4,1){9}}
\put(5,5){\thicklines\vector(4,1){22}}

\put(5,5){\thicklines\line(4,1){29}}
\multiput(5,5)(2,0.5){15}{\line(0,1){1}}

\put(27,34.5){\line(-1,0){23}}
\put(27,34.5){\line(0,-1){30.5}}

\put(5,5){\thicklines\vector(3,4){22}}

\put(0,0){$0$}

\put(7,27){$\bm{a}_{2}$}

\put(7,37){$x_{2}\bm{a}_{2}$}

\put(13,10){$\bm{a}_{1}$}

\put(29,15){$x_{1}\bm{a}_{1}$}


\end{picture}

\end{center}
\caption{Linear span of vectors in $\mathbb{R}_{\max,+}^{2}$ (left) and in $\mathbb{R}_{\max,\times}^{2}$ (right).}\label{F-LC}
\end{figure}
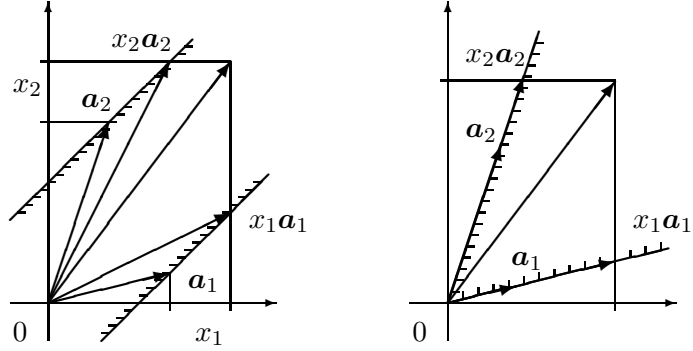

The linear span of vectors $\bm{a}_{1}$ and $\bm{a}_{2}$ in the idempotent space $\mathbb{R}_{\max,+}^{2}$ is a strip bounded by the lines drawn through the end points of the vectors (see Fig.~\ref{F-LC}, left). In $\mathbb{R}_{\max,\times}^{2}$, the linear span has the form of a cone (Fig.~\ref{F-LC}, right).

A system of vectors $\bm{a}_{1},\ldots,\bm{a}_{n}$ is linearly dependent if at least one its vector is linearly dependent on others, and it is linear independent otherwise.

A system of nonzero vectors $\bm{a}_{1},\ldots,\bm{a}_{n}$ is a minimal generating system for a vector $\bm{b}$, if $\bm{b}$ is linearly dependent on the system and independent of any of its subsystems.

Let us verify that if vectors $\bm{a}_{1},\ldots,\bm{a}_{n}$ are a minimal generating system for a vector $\bm{b}$, then representation of $\bm{b}$ as a linear combination of $\bm{a}_{1},\ldots,\bm{a}_{n}$ is unique. Suppose there are two linear combinations
$$
\bm{b}=x_{1}\bm{a}_{1}\oplus\cdots\oplus x_{n}\bm{a}_{n}=x_{1}^{\prime}\bm{a}_{1}\oplus\cdots\oplus x_{n}^{\prime}\bm{a}_{n},
$$
where $x_{i}^{\prime}\ne x_{i}$, say $x_{i}^{\prime}<x_{i}$, for some index $i=1,\ldots,n$.

Assuming, for the sake of simplicity, that the vector $\bm{a}_{i}$ is regular, we have $\bm{b}\geq x_{i}\bm{a}_{i}>x_{i}^{\prime}\bm{a}_{i}$. Therefore, the term $x_{i}^{\prime}\bm{a}_{i}$ does not affect $\bm{b}$ and so may be omitted, which contradicts with the minimality of the system $\bm{a}_{1},\ldots,\bm{a}_{n}$.

\subsection{Distance Function}

For any vector $\bm{a}\in\mathbb{X}^{m}$, we introduce its support as the index set
$$
\mathop\mathrm{supp}(\bm{a})
=
\{i|a_{i}\ne\mathbb{0},i=1,\ldots,m\}.
$$

The distance between nonzero vectors $\bm{a},\bm{b}\in\mathbb{X}^{m}$ with $\mathop\mathrm{supp}(\bm{a})=\mathop\mathrm{supp}(\bm{b})$ is defined by a function
\begin{equation}
\rho(\bm{a},\bm{b})
=
\bigoplus_{i\in\mathop\mathrm{supp}(\bm{a})}\left(b_{i}^{-1} a_{i}\oplus a_{i}^{-1} b_{i}\right)
=
\bm{b}^{-}\bm{a}\oplus\bm{a}^{-}\bm{b}.
\label{E-rhoab}
\end{equation}

We put $ \rho(\bm{a},\bm{b})=\infty$ when $\mathop\mathrm{supp}(\bm{a})\ne\mathop\mathrm{supp}(\bm{b})$, and $\rho(\bm{a},\bm{b})=\mathbb{1}$ if $\bm{a}=\bm{b}=\mathbb{0}$.

Note that the function $\rho$ in $\mathbb{R}_{\max,+}^{m}$ coincides for all vectors $\bm{a},\bm{b}\in\mathbb{R}^{m}$ with the Chebyshev metric
$$
\rho_{\infty}(\bm{a},\bm{b})
=
\max_{1\leq i\leq m}|b_{i}-a_{i}|.
$$

\section{Evaluation of Distances}

Let $\bm{a}_{1},\ldots,\bm{a}_{n}\in\mathbb{X}^{m}$ be given vectors. We denote by $A=(\bm{a}_{1},\ldots,\bm{a}_{n})$ a matrix having the vectors as columns, and by  $\mathcal{A}=\mathop\mathrm{span}\{\bm{a}_{1},\ldots,\bm{a}_{n}\}$ a linear span of the vectors.

Take a vector $\bm{d}\in\mathbb{X}^{m}$ and consider the problem of computing the distance from $\bm{d}$ to $\mathcal{A}$ defined as
$$
\rho(\mathcal{A},\bm{d})
=
\min_{\bm{a}\in\mathcal{A}}\rho(\bm{a},\bm{d}).
$$

As another problem of interest, we examine the distance from $\bm{d}$ to the sets (half-spaces)
$$
\mathcal{A}_{1}
=
\{\bm{a}\in\mathcal{A}|\bm{a}\leq\bm{d}\},
\qquad
\mathcal{A}_{2}
=
\{\bm{a}\in\mathcal{A}|\bm{a}\geq\bm{d}\}.
$$ 

Taking into account that every vector $\bm{a}\in\mathcal{A}$ can be represented as $\bm{a}=A\bm{x}$ for some vector $\bm{x}\in\mathbb{X}^{n}$, we arrive at the problem of calculating
\begin{equation}
\rho(\mathcal{A},\bm{d})
=
\min_{\bm{x}\in\mathbb{X}^{n}}\rho(A\bm{x},\bm{d}).
\label{E-rhoAd}
\end{equation}

Suppose $\bm{d}=\mathbb{0}$. Considering that $\mathcal{A}$ always contains the zero vector, we obviously get $\rho(\mathcal{A},\bm{d})=\mathbb{1}$.

Let some of the vectors $\bm{a}_{1},\ldots,\bm{a}_{n}$ be zero. Since zero vectors do not affect the linear span $\mathcal{A}$, they can be removed with no change of distances. When all vectors are zero and thus $\mathcal{A}=\{\mathbb{0}\}$, we have $\rho(\mathcal{A},\bm{d})=\mathbb{0}$ if $\bm{d}=\mathbb{0}$, and $\rho(\mathcal{A},\bm{d})=\infty$ otherwise.

From here on we assume that $\bm{d}\ne\mathbb{0}$ and $\bm{a}_{i}\ne\mathbb{0}$ for all $i=1,\ldots,n$, and so the matrix $A$ is column regular.

Suppose the vector $\bm{d}=(d_{i})$ may have zero components and so be irregular. For the matrix $A=(a_{ij})$, we introduce a matrix $\widehat{A}=(\widehat{a}_{ij})$ as follows. We define two sets of indices $I=\{i|d_{i}=\mathbb{0}\}$ and $J=\{j|a_{ij}>\mathbb{0}, i\in I\}$, and then determine the entries in $\widehat{A}$ according to the conditions
$$
\widehat{a}_{ij}
=
\begin{cases}
\mathbb{0},	& \text{if $i\notin I$ and $j\in J$}; \\
a_{ij},			& \text{otherwise}.
\end{cases}
$$

The matrix $A$ may differ from $\widehat{A}$ only in those columns that have nonzero intersections with the rows corresponding to zero components in $\bm{d}$. In the matrix $\widehat{A}$, these columns have all entries that are not located at the intersections set to zero. The matrix $\widehat{A}$ with the above properties and the vector $\bm{d}$ are said to be consistent with each other. 

Note that when $\bm{d}$ is regular, we have $\widehat{A}=A$; that is, the matrix $\widehat{A}$ obtained from the matrix $A$ to provide consistency with $\bm{d}$ appears to coincide with $A$.

\begin{proposition}\label{P-AxdAxd}
For all $\bm{x}$ it holds that
$$
\rho(A\bm{x},\bm{d})
=
\rho(\widehat{A}\bm{x},\bm{d}).
$$
\end{proposition}
\begin{proof}
With a regular $\bm{d}$ the statement becomes trivial and so assume $\bm{d}\ne\mathbb{0}$ to have zero components.

Suppose that $\rho(A\bm{x},\bm{d})<\infty$, which occurs only under the condition $\mathop\mathrm{supp}(A\bm{x})=\mathop\mathrm{supp}(\bm{d})$. The fulfillment of the condition is equivalent to equalities $a_{i1}x_{1}\oplus\cdots\oplus a_{in}x_{n}=\mathbb{0}$ that must be true whenever $d_{i}=\mathbb{0}$. To provide the equalities, we put $x_{j}=\mathbb{0}$ for all indices $j$ such that $a_{ij}\ne\mathbb{0}$ for at least one index $i$ with $d_{i}=\mathbb{0}$. In this case, replacing $A$ with $\widehat{A}$ leaves the value of $\rho(A\bm{x},\bm{d})<\infty$ unchanged.

Since the condition $\mathop\mathrm{supp}(A\bm{x})\ne\mathop\mathrm{supp}(\bm{d})$ implies $\mathop\mathrm{supp}(\widehat{A}\bm{x})\ne\mathop\mathrm{supp}(\bm{d})$ and vice versa, the statement is also true when $\rho(A\bm{x},\bm{d})=\infty$.
\end{proof}

With the above result, we may now concentrate only on the problems when $A$ is consistent with $\bm{d}$.

In order to describe the solution of problem \eqref{E-rhoAd}, we need the following notation. For any consistent matrix $A$ and vector $\bm{d}$, we define a residual value
$$
\Delta_{A}(\bm{d})
=
\sqrt{(A(\bm{d}^{-}A)^{-})^{-}\bm{d}}
$$
if $A$ is row regular, and $\Delta_{A}(\bm{d})=\infty$ otherwise.

In what follows, we drop subscripts and arguments in $\Delta_{A}(\bm{d})$ and write $\Delta$ if no confusion arises.
 
Below we find the solution when the vector $\bm{d}$ is regular and then extend this result to irregular vectors.
 
\subsection{Regular Vector}

Suppose that $\bm{d}$ is a regular vector. First we verify that the minimum of $\rho(A\bm{x},\bm{d})$ over $\mathbb{X}^{n}$ in \eqref{E-rhoAd} can be found by examining only regular vectors $\bm{x}\in\mathbb{X}_{+}^{n}$.

\begin{proposition}
If a vector $\bm{d}$ is regular, then
$$
\rho(\mathcal{A},\bm{d})
=
\min_{\bm{x}\in\mathbb{X}_{+}^{n}}\rho(A\bm{x},\bm{d}).
$$
\end{proposition}
\begin{proof}
Take a vector $\bm{y}=A\bm{x}$ such that $\rho(A\bm{x},\bm{d})$ achieves the minimum value. If $\bm{y}$ is irregular and so has zero components, then $\mathop\mathrm{supp}(\bm{y})\ne\mathop\mathrm{supp}(\bm{d})$, and thus $\rho(A\bm{x},\bm{d})=\infty$ for all $\bm{x}$, including regular vectors $\bm{x}$.

Suppose $\bm{y}=(y_{1},\ldots,y_{m})^{T}$ is regular. Assume a corresponding vector $\bm{x}$ to have a zero component, say $x_{j}=\mathbb{0}$. Now we define a set $I=\{i|a_{ij}>\mathbb{0}\}\ne\emptyset$ and find a number $\varepsilon=\min\{a_{ij}^{-1}y_{i}|i\in I\}>\mathbb{0}$.

It remains to note that with $x_{j}=\varepsilon$ in place of $x_{j}=\mathbb{0}$, all components of $\bm{y}$ together with the minimum value of $\rho(A\bm{x},\bm{d})$ remain unchanged. Therefore, to get the minimum it is suffice to examine only regular vectors $\bm{x}\in\mathbb{X}_{+}^{n}$.
\end{proof}

The next statement reveals the meaning of the residual value $\Delta=\Delta_{A}(\bm{d})$ in terms of distances.
\begin{lemma}\label{L-rhoAdr}
If a vector $\bm{d}$ is regular, then it holds that
$$
\rho(\mathcal{A},\bm{d})
=
\Delta,
$$
where the minimum is attained at
$$
\bm{x}
=
\Delta(\bm{d}^{-}A)^{-}.
$$
\end{lemma}
\begin{proof}
Suppose the matrix $A$ is not row regular. Then we have $\mathop\mathrm{supp}(A\bm{x})\ne\mathop\mathrm{supp}(\bm{d})$ and $\rho(\mathcal{A},\bm{d})=\infty$. Since, by definition, $\Delta=\infty$, the statement is true in this case.

Let $A$ be row regular. Taking into account \eqref{E-rhoab} and \eqref{E-rhoAd}, we arrive at an optimization problem to find
$$
\min_{\bm{x}\in\mathbb{X}_{+}^{n}}\ (\bm{d}^{-}A\bm{x}\oplus(A\bm{x})^{-}\bm{d}).
$$ 

Take any vector $\bm{y}=A\bm{x}$ such that $\bm{x}>\mathbb{0}$, and define
$$
r
=
\bm{d}^{-}A\bm{x}\oplus(A\bm{x})^{-}\bm{d}>\mathbb{0}.
$$

From the definition of $r$, we have two inequalities
$$
r\geq\bm{d}^{-}A\bm{x},
\qquad
r\geq(A\bm{x})^{-}\bm{d}.
$$

Right multiplication of the first inequality by $\bm{x}^{-}$ together with \eqref{I-xyxyI} give $r\bm{x}^{-}\geq\bm{d}^{-}A\bm{x}\bm{x}^{-}\geq\bm{d}^{-}A$. Then we obtain $\bm{x}\leq r(\bm{d}^{-}A)^{-}$ and $(A\bm{x})^{-}\geq r^{-1}(A(\bm{d}^{-}A)^{-})^{-}$.

Further substitution into the second inequality results in
$$
r
\geq
r^{-1}(A(\bm{d}^{-}A)^{-})^{-}\bm{d}
=
r^{-1}\Delta^{2},
$$
whence it follows directly that $r\geq\Delta$.

It remains to verify the equality $r=\Delta$ when we take $\bm{x}=\Delta(\bm{d}^{-}A)^{-}$. Indeed, substitution of this vector $\bm{x}$ gives
$$
r
=
\Delta\bm{d}^{-}A(\bm{d}^{-}A)^{-}\oplus\Delta^{-1}(A(\bm{d}^{-}A)^{-})^{-}\bm{d}
=
\Delta.
$$

Finally note that the above vector $\bm{x}$ corresponds to the vector $\bm{y}=\Delta A(\bm{d}^{-}A)^{-}\in\mathcal{A}$.
\end{proof}

Examples of a subspace $\mathcal{A}=\mathop\mathrm{span}(\bm{a}_{1},\bm{a}_{2})$ and a vector $\bm{d}$ in the idempotent space $\mathbb{R}_{\max,+}^{2}$ are given in Fig.~\ref{F-Lb}.
\begin{figure}[ht]
\setlength{\unitlength}{1mm}
\begin{center}

\begin{picture}(45,50)

\put(0,5){\vector(1,0){45}}
\put(6,0){\vector(0,1){50}}

\put(6,5){\thicklines\vector(1,4){3}}

\put(6,5){\thicklines\vector(2,-1){5}}

\put(1,9){\line(1,1){37}}
\multiput(2,10)(1,1){36}{\line(1,0){1}}
\put(1,9){\thicklines\line(1,1){37}}

\put(8.5,0){\line(1,1){33}}
\multiput(9,0.5)(1,1){33}{\line(-1,0){1}}
\put(8.5,0){\thicklines\line(1,1){33}}

\put(6,5){\thicklines\vector(3,4){18}}

\put(13,0){$\bm{a}_{1}$}
\put(8,23){$\bm{a}_{2}$}
\put(29,32){$\bm{d}$}

\put(38,39){$\mathcal{A}$}

\put(12,44){$\Delta=\mathbb{1}$}

\end{picture}
\hspace{15\unitlength}
\begin{picture}(45,50)

\put(0,5){\vector(1,0){45}}
\put(6,0){\vector(0,1){50}}

\put(6,5){\thicklines\vector(1,4){3}}

\put(6,5){\thicklines\vector(2,-1){5}}

\put(1,9){\line(1,1){37}}
\multiput(2,10)(1,1){36}{\line(1,0){1}}
\put(1,9){\thicklines\line(1,1){37}}

\put(8.5,0){\line(1,1){33}}
\multiput(9,0.5)(1,1){33}{\line(-1,0){1}}
\put(8.5,0){\thicklines\line(1,1){33}}

\put(28.5,20){\line(1,-1){9.5}}

\put(6,5){\thicklines\vector(3,2){22.5}}

\put(6,5){\thicklines\line(6,1){32}}
\put(36,10){\thicklines\vector(4,1){2}}

\multiput(28.5,20)(0,-3){5}{\line(0,-1){2}}
\put(28.5,5){\line(0,-1){1}}

\multiput(38,10.5)(0,-2.9){2}{\line(0,-1){2}}
\put(38,5){\line(0,-1){1}}

\put(13,0){$\bm{a}_{1}$}
\put(8,23){$\bm{a}_{2}$}
\put(39,11){$\bm{d}$}

\put(38,39){$\mathcal{A}$}

\put(25,23){$\bm{y}$}

\put(12,44){$\Delta>\mathbb{1}$}

\put(32,0){$\Delta$}

\end{picture}

\end{center}
\caption{A linear span $\mathcal{A}$ and a vector $\bm{d}$ in $\mathbb{R}_{\max,+}^{2}$ when $\Delta=\mathbb{1}$ (top) and $\Delta>\mathbb{1}$ (bottom).}\label{F-Lb}
\end{figure}
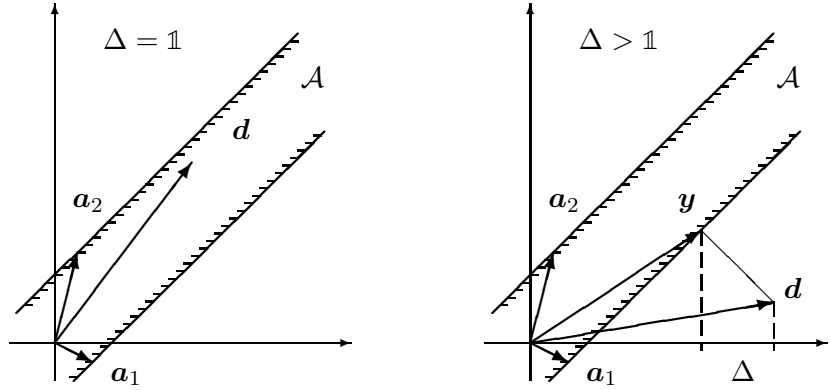

Now we turn to evaluation of the distance from the vector $\bm{d}$ to the half-spaces $\mathcal{A}_{1}$ and $\mathcal{A}_{2}$.

\begin{lemma}\label{L-rhoA1A2dr}
If a vector $\bm{d}$ is regular, then it holds that
\begin{align*}
\rho(\mathcal{A}_{1},\bm{d})
&=
\min_{A\bm{x}\leq\bm{d}}\rho(A\bm{x},\bm{d})
=
\Delta^{2},
\\
\rho(\mathcal{A}_{2},\bm{d})
&=
\min_{A\bm{x}\geq\bm{d}}\rho(A\bm{x},\bm{d})
=
\Delta^{2},
\end{align*}
where the minimum is respectively attained at
$$
\bm{x}_{1}
=
(\bm{d}^{-}A)^{-},
\qquad
\bm{x}_{2}
=
\Delta^{2}(\bm{d}^{-}A)^{-}.
$$
\end{lemma}
\begin{proof}
Similarly as in Lemma~\ref{L-rhoAdr} we can verify the equality $\rho(\mathcal{A}_{1},\bm{d})=\rho(\mathcal{A}_{2},\bm{d})=\Delta^{2}$ provided that $A$ is not a row regular matrix. Let us show that the equality remains valid when the matrix $A$ is row regular.

By multiplying the inequality $A\bm{x}\leq\bm{d}$ by $\bm{x}^{-}$ from the right and applying \eqref{I-xyxyI} we get $A\leq A\bm{x}\bm{x}^{-}\leq\bm{d}\bm{x}^{-}$. Further multiplication by $\bm{d}^{-}$ from the left results in the inequality $\bm{d}^{-}A\leq\bm{x}^{-}$, which then gives $\bm{x}\leq(\bm{d}^{-}A)^{-}$.
 
Therefore, for any vector $A\bm{x}\in\mathcal{A}_{1}$, we have
$$
\rho(A\bm{x},\bm{d})
=
(A\bm{x})^{-}\bm{d}
\geq
(A(\bm{d}^{-}A)^{-})^{-}\bm{d}
=
\Delta^{2}.
$$

It is clear that $\rho(A\bm{x}_{1},\bm{d})=\Delta^{2}$ if $\bm{x}_{1}=(\bm{d}^{-}A)^{-}$.

Consider an arbitrary vector $A\bm{x}\in\mathcal{A}_{2}$. Application of \eqref{I-AxAx} to the vector $(\bm{d}^{-}A)^{-}$ yields $\bm{d}^{-}A\geq(A(\bm{d}^{-}A)^{-})^{-}A$. Taking into account the condition that $A\bm{x}\geq\bm{d}$, we further have
$\bm{d}^{-}A\bm{x}\geq(A(\bm{d}^{-}A)^{-})^{-}A\bm{x}\geq(A(\bm{d}^{-}A)^{-})^{-}\bm{d}$.

Now we can conclude that for $A\bm{x}\in\mathcal{A}_{2}$, it holds
$$
\rho(A\bm{x},\bm{d})
=
\bm{d}^{-}A\bm{x}
\geq
(A(\bm{d}^{-}A)^{-})^{-}\bm{d}
=
\Delta^{2}.
$$

It remains to see that substitution $\bm{x}_{2}=\Delta^{2}(\bm{d}^{-}A)^{-}$ gives $\rho(A\bm{x}_{2},\bm{d})=\Delta^{2}\bm{d}^{-}A(\bm{d}^{-}A)^{-}=\Delta^{2}$.

Note that when $\Delta<\infty$ the minimum distance from $\bm{d}$ to the half-spaces $\mathcal{A}_{1}$ and $\mathcal{A}_{2}$ is achieved at the respective vectors $\bm{y}_{1}=A(\bm{d}^{-}A)^{-}$ and $\bm{y}_{2}=\Delta^{2}A(\bm{b}^{-}A)^{-}$.
\end{proof}

A geometric illustration of the above result in the idempotent space $\mathbb{R}_{\max,+}^{2}$ is given in Fig.~\ref{F-LLb}.
\begin{figure}[ht]
\setlength{\unitlength}{1mm}
\begin{center}

\begin{picture}(45,50)

\put(0,5){\vector(1,0){45}}
\put(6,0){\vector(0,1){50}}

\put(6,5){\thicklines\vector(1,4){3}}

\put(6,5){\thicklines\vector(2,-1){5}}

\put(1,9){\line(1,1){37}}
\multiput(2,10)(1,1){19}{\line(1,0){1}}
\put(1,9){\thicklines\line(1,1){20}}

\multiput(24,32)(1,1){15}{\line(1,0){1}}
\put(24,32){\thicklines\line(1,1){15}}

\put(8.5,0){\line(1,1){32}}
\multiput(9,0.5)(1,1){16}{\line(-1,0){1}}
\put(8.5,0){\thicklines\line(1,1){15.5}}

\multiput(38.5,30)(1,1){5}{\line(-1,0){1}}
\put(37.5,29){\thicklines\line(1,1){5}}

\put(6,5){\thicklines\vector(3,4){18}}

\put(21,29){\line(1,0){16.5}}
\put(21,29){\thicklines\line(1,0){16.5}}
\multiput(21,29)(1,0){3}{\line(0,-1){1}}
\multiput(24,29)(1,0){14}{\line(0,1){1}}

\put(24,15.5){\line(0,1){16.5}}
\put(24,15.5){\thicklines\line(0,1){16.5}}
\multiput(24,15.5)(0,1){14}{\line(-1,0){1}}
\multiput(24,30)(0,1){3}{\line(1,0){1}}

\put(13,0){$\bm{a}_{1}$}
\put(8,23){$\bm{a}_{2}$}
\put(29,32){$\bm{d}$}

\put(0,0){$\mathcal{A}_{1}$}
\put(38,39){$\mathcal{A}_{2}$}

\put(12,44){$\Delta=\mathbb{1}$}

\end{picture}
\hspace{15\unitlength}
\begin{picture}(45,50)

\put(0,5){\vector(1,0){45}}
\put(6,0){\vector(0,1){50}}

\put(6,5){\thicklines\vector(1,4){3}}

\put(6,5){\thicklines\vector(2,-1){5}}

\put(1,9){\line(1,1){38}}
\multiput(1.5,9.5)(1,1){3}{\line(1,0){1}}
\put(1,9){\thicklines\line(1,1){3.6}}

\multiput(36,44)(1,1){4}{\line(1,0){1}}
\put(36,44){\thicklines\line(1,1){3.6}}

\put(8.5,0){\line(1,1){32}}

\multiput(9,0.5)(1,1){12}{\line(-1,0){1}}
\put(8.5,0){\thicklines\line(1,1){12.5}}

\multiput(37.5,29)(1,1){5}{\line(-1,0){1}}
\put(36,27.5){\thicklines\line(1,1){6}}

\multiput(4.5,12.5)(1,0){16}{\line(0,-1){1}}
\put(4.5,12.5){\thicklines\line(1,0){16.5}}

\multiput(36,44)(0,-1){16}{\line(1,0){1}}
\put(36,44){\thicklines\line(0,-1){16.5}}

\put(6,5){\thicklines\vector(4,3){30}}

\put(6,5){\thicklines\vector(2,1){15}}

\put(6,5){\thicklines\vector(4,1){30}}

\put(36,12.5){\line(-1,0){32}}
\put(36,12.5){\line(0,1){31.5}}

\multiput(21.1,12.5)(0,-3){3}{\line(0,-1){2}}
\put(21.1,5){\line(0,-1){1}}

\multiput(36,12.6)(0,-3){3}{\line(0,-1){2}}
\put(36,5){\line(0,-1){1}}

\put(14,0){$\bm{a}_{1}$}
\put(8,23){$\bm{a}_{2}$}
\put(38,12){$\bm{d}$}

\put(0,0){$\mathcal{A}_{1}$}
\put(39,39){$\mathcal{A}_{2}$}

\put(26,15){$\bm{y}_{1}$}
\put(30,30){$\bm{y}_{2}$}

\put(12,44){$\Delta>\mathbb{1}$}

\put(27,0){$\Delta^{2}$}

\end{picture}

\end{center}
\caption{The sets $\mathcal{A}_{1}$ and $\mathcal{A}_{2}$, and the vector $\bm{d}$ in $\mathbb{R}_{\max,+}^{2}$ when $\Delta=\mathbb{1}$ (top) and $\Delta>\mathbb{1}$ (bottom).}\label{F-LLb}
\end{figure}
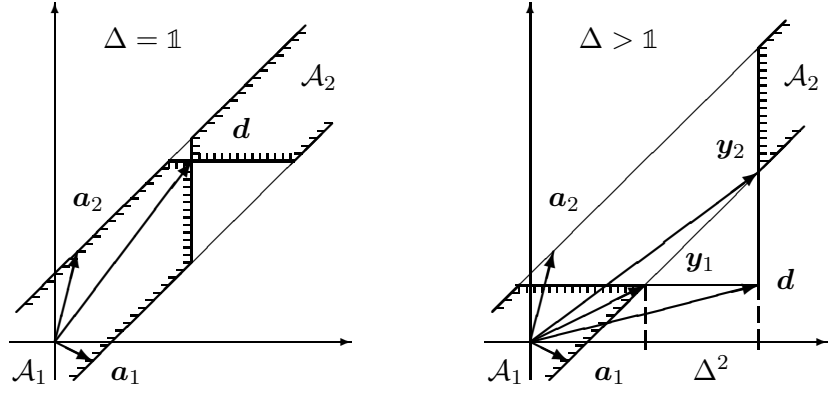

\subsection{Arbitrary Nonzero Vector}

Now we examine the distance between the linear span $\mathcal{A}$ and an arbitrary vector $\bm{d}\ne\mathbb{0}$.

\begin{theorem}\label{T-Lb}
For any vector $\bm{d}\ne\mathbb{0}$ it holds that
$$
\rho(\mathcal{A},\bm{d})
=
\min_{\bm{x}\in\mathbb{X}_{+}^{n}}\rho(A\bm{x},\bm{d})
=
\Delta,
$$
where the minimum is attained at $\bm{x}=\Delta(\bm{d}^{-}A)^{-}$.
\end{theorem}
\begin{proof}
Note that for the case when $\bm{d}$ is regular, the proof is given in Lemma~\ref{L-rhoAdr}. Now we suppose that the vector $\bm{d}\ne\mathbb{0}$ has zero components. Due to Proposition~\ref{P-AxdAxd}, it will suffice to examine only the case when $A$ and $\bm{d}$ are consistent.

Let us define the sets of indices $I=\{i|d_{i}=\mathbb{0}\}$ and $J=\{j|a_{ij}>\mathbb{0}, i\in I\}$. In order to provide the minimum of $\rho(A\bm{x},\bm{d})$, we must put $x_{j}=\mathbb{0}$ for all $j\in J$. This makes it possible to exclude from consideration all  components of $\bm{d}$ and the rows of $A$ with indices in $I$, as well as all columns of $A$ with indices in $J$. By eliminating these elements, we obtain a new matrix $A^{\prime}$ and a new vector $\bm{d}^{\prime}$.

Denote the linear span of the columns in $A^{\prime}$ by $\mathcal{A}^{\prime}$. Considering that the vector $\bm{d}^{\prime}$ has no zero components, we apply Lemma~\ref{L-rhoAdr} to get
$$
\rho(\mathcal{A},\bm{d})
=
\rho(\mathcal{A}^{\prime},\bm{d}^{\prime})
=
\Delta_{A^{\prime}}(\bm{d}^{\prime})
=
\Delta^{\prime}.
$$

Furthermore, we note that the minimum $\rho(A^{\prime}\bm{x}^{\prime},\bm{d}^{\prime})$ is attained if  $\bm{x}^{\prime}=\Delta^{\prime}(\bm{d}^{\prime-}A^{\prime})^{-}$, where $\bm{x}^{\prime}$ is a vector of order less than $n$.

The matrix $A$ differs from $A^{\prime}$ only in that it has extra zero rows and columns. Clearly, both matrices appear to be row regular or irregular simultaneously.

Suppose that both matrices are row regular. Taking into account that the vector $\bm{d}^{\prime}$ is obtained from $\bm{d}$ by removing zero components, we have
$$
\Delta^{\prime}
=
\sqrt{(A^{\prime}(\bm{d}^{\prime-}A^{\prime})^{-})^{-}\bm{d}^{\prime}}
=
\sqrt{(A(\bm{d}^{-}A)^{-})^{-}\bm{d}}
=
\Delta.
$$

Since the optimal vector $\bm{x}$ differs from $\bm{x}^{\prime}$ only in extra zero components, we conclude that $\rho(A\bm{x},\bm{d})$ achieves minimum at $\bm{x}=\Delta(\bm{d}^{-}A)^{-}$.
\end{proof}

Using the same proof scheme as above, it is not difficult to extend the result of Lemma~\ref{L-rhoA1A2dr} as follows.
\begin{lemma}\label{L-rhoA1A2d}
For any vector $\bm{d}\ne\mathbb{0}$, it holds that
\begin{align*}
\rho(\mathcal{A}_{1},\bm{d})
&=
\min_{A\bm{x}\leq\bm{d}}\rho(A\bm{x},\bm{d})
=
\Delta^{2},
\\
\rho(\mathcal{A}_{2},\bm{d})
&=
\min_{A\bm{x}\geq\bm{d}}\rho(A\bm{x},\bm{d})
=
\Delta^{2},
\end{align*}
where the minimums are respectively attained at
$$
\bm{x}_{1}
=
(\bm{d}^{-}A)^{-},
\qquad
\bm{x}_{2}
=
\Delta^{2}(\bm{d}^{-}A)^{-}.
$$
\end{lemma}

To conclude this section, let us formulate a direct consequence of Theorem~\ref{T-Lb}. First note that the residual $\Delta$ satisfies the condition $\Delta\geq\mathbb{1}$. The equality $\Delta=\mathbb{1}$ means that the vector $\bm{d}$ belongs to the linear span $\mathcal{A}=\mathop\mathrm{span}\{\bm{a}_{1},\ldots,\bm{a}_{n}\}$, whereas the inequality $\Delta>\mathbb{1}$ implies that $\bm{d}$ is outside $\mathcal{A}$. In other words, the following assertion is valid.
\begin{lemma}\label{L-LH}
A vector $\bm{d}$ belongs to a linear span of columns of a matrix $A$ if and only if $\Delta=\mathbb{1}$, and if so it holds that $\bm{d}=A\bm{x}$, where $\bm{x}=(\bm{d}^{-}A)^{-}$.
\end{lemma}

In the next sections, we consider applications of the result to analysis of linear dependence and to solution of linear equations and inequalities.

\section{Linear Dependence}

First we give conditions for a vector $\bm{d}\in\mathbb{X}^{m}$ to be linearly dependent on vectors $\bm{a}_{1},\ldots,\bm{a}_{n}\in\mathbb{X}^{m}$, or equivalently, to admit a representation in the form of a linear combination $\bm{d}=x_{1}\bm{a}_{1}\oplus\cdots\oplus x_{n}\bm{a}_{n}$.

We define the matrix $A=(\bm{a}_{1},\ldots,\bm{a}_{n})$ and then calculate the residual $\Delta=\Delta_{A}(\bm{d})=\sqrt{(A(\bm{d}^{-}A)^{-})^{-}\bm{d}}$.

As a slight modification of Lemma~\ref{L-LH}, we arrive at the following statement.
\begin{lemma}\label{L-LD}
A vector $\bm{d}$ is linearly dependent on vectors $\bm{a}_{1},\ldots,\bm{a}_{n}$ if and only if $\Delta=\mathbb{1}$.
\end{lemma}

Now we formulate a formal criterion that a system $\bm{a}_{1},\ldots,\bm{a}_{n}$ is linearly independent. We use the notation $A_{i}=(\bm{a}_{1},\ldots,\bm{a}_{i-1},\bm{a}_{i+1},\ldots,\bm{a}_{n})$ to represent a matrix obtained from $A$ by removing column $i$, and introduce
$$
\delta(A)=\min_{1\leq i\leq n}\Delta_{A_{i}}(\bm{a}_{i}).
$$

\begin{lemma}\label{L-LIC}
The system of vectors $\bm{a}_{1},\ldots,\bm{a}_{n}$ is linearly independent if and only if $\delta(A)>\mathbb{1}$.
\end{lemma}
\begin{proof}
Clearly, the condition $\delta(A)>\mathbb{1}$ involves that $\Delta_{A_{i}}(\bm{a}_{i})>\mathbb{1}$ for all $i=1,\ldots,n$. It follows from Lemma~\ref{L-LD} that in this case none of the vectors $\bm{a}_{1},\ldots,\bm{a}_{n}$ is a linear combination of others, and so the system of vectors is linearly independent.
\end{proof}

Let $\bm{a}_{1},\ldots,\bm{a}_{n}$ and $\bm{b}_{1},\ldots,\bm{b}_{k}$ be two systems of nonzero vectors. These systems are considered to be equivalent if each vector of one system is a linear combination of vectors of the other system.

Consider a system $\bm{a}_{1},\ldots,\bm{a}_{n}$ that can include linearly dependent vectors. To construct an equivalent independent system (a basis of the system), we implement a sequential procedure that examines the vectors one by one to decide whether to remove a vector from the system or not.

At each step $i=1,\ldots,n$, the vector $\bm{a}_{i}$ is removed if $\Delta_{\widetilde{A}_{i}}(\bm{a}_{i})=\mathbb{1}$, where the matrix $\widetilde{A}_{i}$ is composed of those columns in $A_{i}$, that are retained after the previous steps. Upon completion of the procedure, we get a new system of vectors $\widetilde{\bm{a}}_{1},\ldots,\widetilde{\bm{a}}_{k}$, where $k\leq n$.

\begin{proposition}
The system $\widetilde{\bm{a}}_{1},\ldots,\widetilde{\bm{a}}_{k}$ is a linearly independent system that is equivalent to $\bm{a}_{1},\ldots,\bm{a}_{n}$.
\end{proposition}
\begin{proof}
According to the way of constructing the system $\widetilde{\bm{a}}_{1},\ldots,\widetilde{\bm{a}}_{k}$, each vector $\widetilde{\bm{a}}_{i}$ coincides with a vector of the initial system $\bm{a}_{1},\ldots,\bm{a}_{n}$. Since at the same time, for each $\bm{a}_{j}$, it holds that $\bm{a}_{j}\in\mathop\mathrm{span}\{\widetilde{\bm{a}}_{1},\ldots,\widetilde{\bm{a}}_{k}\}$, both systems are equivalent. Finally, due to Lemma~\ref{L-LIC}, the system $\widetilde{\bm{a}}_{1},\ldots,\widetilde{\bm{a}}_{k}$ is linearly independent.
\end{proof}

\section{Linear Equations and Inequalities}

Suppose there are given a matrix $A\in\mathbb{X}^{m\times n}$ and a vector $\bm{d}\in\mathbb{X}^{m}$. Consider problems of finding an unknown vector $\bm{x}\in\mathbb{X}^{n}$ to satisfy the equation
\begin{equation}
A\bm{x}
=
\bm{d},
\label{E-Axd}
\end{equation}
and the inequality
\begin{equation}
A\bm{x}
\leq
\bm{d}.
\label{I-Axd}
\end{equation}

In what follows, we assume that the matrix $A$ is already put into a form that is consistent with the vector $\bm{d}$, and use the notation $\Delta=\Delta_{A}(\bm{d})=\sqrt{(A(\bm{d}^{-}A)^{-})^{-}\bm{d}}$.

If a matrix $A=(\bm{a}_{1},\ldots,\bm{a}_{n})$ has a zero column, say $\bm{a}_{i}$, then the solution of equation \eqref{E-Axd} reduces to that of an equation that is obtained from \eqref{E-Axd} by removing the component $x_{i}$ in the vector $\bm{x}$ together with eliminating the column $\bm{a}_{i}$ in $A$. Each solution of the reduced equation causes equation \eqref{E-Axd} to have a set of solutions, where $x_{i}$ takes all values in $\mathbb{X}$. The solution of inequality \eqref{I-Axd} with a matrix $A$ having a zero column reduces in the same way.

Suppose that $A=\mathbb{0}$. In this case, any vector $\bm{x}\in\mathbb{X}^{n}$ is a solution of \eqref{E-Axd} provided that $\bm{d}=\mathbb{0}$, and there is no solution otherwise. The solution of \eqref{I-Axd} is any vector $\bm{x}\in\mathbb{X}^{n}$.

If $\bm{d}=\mathbb{0}$, then both equation \eqref{E-Axd} and inequality \eqref{I-Axd} have a trivial solution $\bm{x}=\mathbb{0}$, which is unique when the matrix $A$ has no zero columns.

From here on we assume that the vector $\bm{d}$ and all columns in the matrix $A$ are nonzero, and so $A$ is column regular.

A solution $\bm{x}_{0}$ of equation \eqref{E-Axd} is called maximal if it holds that $\bm{x}\leq\bm{x}_{0}$ for any solution $\bm{x}$.

The next result gives a complete solution of inequality \eqref{I-Axd}.

\begin{lemma}\label{L-xdA}
For any column regular matrix $A$ and vector $\bm{b}\ne\mathbb{0}$, the solution of inequality \eqref{I-Axd} exists and is given by
\begin{equation}
\bm{x}
\leq
(\bm{d}^{-}A)^{-}.
\label{I-xdA}
\end{equation}
\end{lemma}
\begin{proof}
Let us ensure that inequalities \eqref{I-xdA} and \eqref{I-Axd} are equivalent to each other. First assume the vector $\bm{d}$ to be regular. In much the same way as in Lemma~\ref{L-rhoA1A2dr}, we verify that inequality \eqref{I-xdA} follows from \eqref{I-Axd}.

Suppose that inequality \eqref{I-xdA} holds. Then we have 
$$
A\bm{x}\leq A(\bm{b}^{-}A)^{-}\leq\bm{b}\bm{b}^{-}A(\bm{b}^{-}A)^{-}=\bm{b},
$$
and thus inequality \eqref{I-Axd} holds as well.

Now assume that the vector $\bm{d}\ne\mathbb{0}$ is not regular. In this case, we use the same proof scheme as in Theorem~\ref{T-Lb} to reduce the problem to that with a regular vector and then apply the above result.
\end{proof}

In the following, we examine conditions for the solution of equation \eqref{E-Axd} to exist and to be unique, and then describe the general solution to the equation. 

\subsection{Existence and Uniqueness of Solution}

Application of previous results brings us to a position to arrive at the next assertion. 
\begin{theorem}\label{T-EAxd}
For any column regular matrix $A$ and nonzero vector $\bm{b}$, the following statements are true:
\begin{enumerate}
\item Equation \eqref{E-Axd} has solutions if and only if $\Delta=\mathbb{1}$.
\item If solvable, the equation has a maximal solution
$$
\bm{x}
=
(\bm{d}^{-}A)^{-}.
$$
\item If all columns in $A$ form a minimal system that generates $\bm{d}$, then the above solution is unique.
\end{enumerate}
\end{theorem}
\begin{proof}
The existence condition and the form of a solution follows from Lemma~\ref{L-LH}. The result of Lemma~\ref{L-xdA} says that this solution is maximal. The uniqueness condition follows from representation of the vector as a unique linear combination of its minimal set of generators.
\end{proof}

Suppose that $\Delta>\mathbb{1}$. In this case equation \eqref{E-Axd} has no solution. However, we can define a pseudo-solution to \eqref{E-Axd} as a solution of the equation
$$
A\bm{x}
=
\Delta A(\bm{d}^{-}A)^{-},
$$
which always exists and takes the form
$$
\bm{x}_{0}
=
\Delta(\bm{d}^{-}A)^{-}.
$$

By Theorem~\ref{T-Lb}, the pseudo-solution yields the minimum deviation between the vectors $\bm{y}=A\bm{x}$ and the vector $\bm{d}$ in the sense of the metric $\rho$. When $\Delta=\mathbb{1}$, the pseudo-solution obviously coincides with the maximum solution.

Consider a problem of finding two vectors $\bm{x}_{1}$ and $\bm{x}_{2}$ that provide the minimum deviation between both sides of \eqref{E-Axd}, while satisfying the respective inequalities
$$
A\bm{x}\leq\bm{d},
\qquad
A\bm{x}\geq\bm{d}.
$$

It follows from Lemma~\ref{L-rhoA1A2d} that the problem has a solution that is given by
$$
\bm{x}_{1}
=
(\bm{d}^{-}A)^{-},
\qquad
\bm{x}_{2}
=
\Delta^{2}(\bm{d}^{-}A)^{-}.
$$

\subsection{General Solution}

To describe a general solution to equation \eqref{E-Axd}, we first give an auxiliary result that solves \eqref{E-Axd} when the vector $\bm{d}$ is linearly dependent on a subset of columns in the matrix $A$.

\begin{lemma}\label{L-AxdGS}
Let $A=(\bm{a}_{1},\ldots,\bm{a}_{n})$ be a matrix, $I$ be a subset of column indices of $A$, and $\bm{d}\in\mathop\mathrm{span}\{\bm{a}_{i}|i\in I\}$.

Then any vector $\bm{x}_{I}=(x_{i})$, where $x_{i}=(\bm{d}^{-}\bm{a}_{i})^{-}$ if $i\in I$, and $x_{i}\leq(\bm{d}^{-}\bm{a}_{i})^{-}$ otherwise, is a solution to \eqref{E-Axd}.
\end{lemma}
\begin{proof}
Since $\bm{d}\in\mathop\mathrm{span}\{\bm{a}_{i}|i\in I\}\subset\mathop\mathrm{span}\{\bm{a}_{1},\ldots,\bm{a}_{n}\}$, there is a solution $\bm{x}_{I}$ of equation \eqref{E-Axd}, and therefore,
$$
\bm{d}
=
A\bm{x}_{I}
=
\bigoplus_{i=1}^{n}x_{i}\bm{a}_{i}
=
\bigoplus_{i\in I}x_{i}\bm{a}_{i}
\oplus
\bigoplus_{i\not\in I}x_{i}\bm{a}_{i}.
$$

Furthermore, the condition $\bm{d}\in\mathop\mathrm{span}\{\bm{a}_{i}|i\in I\}$ yields an equality
$$
\bm{d}
=
\bigoplus_{i\in I}x_{i}\bm{a}_{i},
$$
which is valid when $x_{i}=(\bm{d}^{-}\bm{a}_{i})^{-}$ for all $i\in I$.

The remaining components with indices $i\not\in I$ must be set so as to satisfy inequalities
$$
\bm{d}
\geq
\bigoplus_{i\not\in I}x_{i}\bm{a}_{i}
\geq
x_{i}\bm{a}_{i}.
$$

It remains to solve the inequalities to conclude that for each $i\not\in I$, we can take any $x_{i}\leq(\bm{d}^{-}\bm{a}_{i})^{-}$.
\end{proof}

Let $I$ be a set of indices of those columns in $A$ that form a minimal generating system for the vector $\bm{d}$. We denote the set of all such index sets $I$ by $\mathcal{I}$. It is clear that $\mathcal{I}\neq\emptyset$ only when equation \eqref{E-Axd} has at least one solution.

By applying Lemma~\ref{L-AxdGS}, we arrive at the following result.

\begin{theorem}\label{T-GS}
The general solution to equation \eqref{E-Axd} is a (possible empty) family of solutions $\{\bm{x}_{I}| I\in\mathcal{I}\}$, where each solution $\bm{x}_{I}=(x_{i})$ is given by
\begin{equation}
\begin{split}
x_{i}
&=
(\bm{d}^{-}\bm{a}_{i})^{-},
\qquad
\text{if $i\in I$},
\\
x_{i}
&\leq
(\bm{d}^{-}\bm{a}_{i})^{-},
\qquad
\text{if $i\not\in I$}.
\end{split}
\label{S-xdaixdai}
\end{equation}
\end{theorem}

Let us examine a case when the family reduces to one solution set. Suppose that the columns in $A$ are linearly independent. Then there may exist only one subset of columns that form a minimal generating system for $\bm{d}$. If the subset coincides with the set of all columns, then the solution reduces to a unique vector $\bm{x}=(\bm{d}^{-}A)^{-}$.

As an illustration, consider equation \eqref{E-Axd} in the idempotent vector space $\mathbb{R}_{\max,+}^{2}$ under the conditions that
$$
A
=
\left(
	\begin{array}{cc}
		a_{11} & a_{12} \\
		a_{21} & a_{22}
	\end{array}
\right),
\qquad
\bm{d}
=
\left(
	\begin{array}{c}
		d_{1} \\
		d_{2}
	\end{array}
\right),
$$
where $a_{11},a_{12},a_{21},a_{22}>\mathbb{0}$ and $d_{1},d_{2}>\mathbb{0}$.

Suppose that $\Delta=\Delta_{A}(\bm{d})=(A(\bm{d}^{-}A)^{-})^{-}\bm{d}=\mathbb{1}$. The maximal solution of the equation takes the form
$$
\bm{x}
=
\left(
	\begin{array}{c}
		(\bm{d}^{-}\bm{a}_{1})^{-1} \\
		(\bm{d}^{-}\bm{a}_{2})^{-1}
	\end{array}
\right)
=
\left(
	\begin{array}{c}
		(d_{1}^{-1}a_{11}\oplus d_{2}^{-1}a_{21})^{-1} \\
		(d_{1}^{-1}a_{12}\oplus d_{2}^{-1}a_{22})^{-1}
	\end{array}
\right).
$$

If the vector $\bm{d}$ is not collinear with any of vectors $\bm{a}_{1}$ and $\bm{a}_{2}$, then the solution is unique (see Fig.~\ref{F-AxdS}, upper left).
\begin{figure}[ht]
\setlength{\unitlength}{1mm}
\begin{center}

\begin{picture}(30,40)

\put(0,5){\vector(1,0){30}}
\put(5,0){\vector(0,1){40}}

\put(5,5){\thicklines\vector(1,4){3}}
\put(9,17){\line(-1,0){5}}

\put(2,11){\thicklines\line(1,1){21}}
\multiput(3,12)(1,1){20}{\line(1,0){1}}

\put(5,5){\thicklines\vector(3,1){12}}
\put(17,9){\line(0,-1){5}}

\put(10,2){\thicklines\line(1,1){16}}
\multiput(11,3)(1,1){16}{\line(-1,0){1}}

\put(5,5){\thicklines\vector(3,4){18}}
\put(23,29){\line(-1,0){19}}
\put(23,29){\line(0,-1){25}}

\put(13,12){$\bm{a}_{1}$}
\put(6,21){$\bm{a}_{2}$}
\put(24,30){$\bm{d}$}

\put(18,1){$x_{1}$}
\put(0,23){$x_{2}$}

\end{picture}
\hspace{10\unitlength}
\begin{picture}(30,40)

\put(0,5){\vector(1,0){30}}
\put(5,0){\vector(0,1){40}}

\put(5,5){\thicklines\vector(1,4){3}}
\put(9,17){\line(-1,0){5}}

\put(2,11){\thicklines\line(1,1){24}}
\multiput(3,12)(1,1){23}{\line(1,0){1}}

\put(5,5){\thicklines\vector(3,1){12}}
\put(17,9){\line(0,-1){5}}

\put(10,2){\thicklines\line(1,1){16}}
\multiput(11,3)(1,1){16}{\line(-1,0){1}}

\put(5,5){\thicklines\vector(2,3){18}}
\put(23,32){\line(-1,0){19}}
\put(23,32){\line(0,-1){28}}

\put(13,12){$\bm{a}_{1}$}
\put(6,21){$\bm{a}_{2}$}
\put(20,34){$\bm{d}$}

\put(18,1){$x_{1}$}
\put(0,24){$x_{2}$}

\end{picture}
\hspace{10\unitlength}
\begin{picture}(30,40)

\put(0,5){\vector(1,0){30}}
\put(5,0){\vector(0,1){40}}

\put(5,5){\thicklines\vector(0,1){9}}
\put(9.5,18.5){\line(-1,0){5.5}}

\put(2.0,11){\thicklines\line(1,1){22}}
\put(2.1,11){\thicklines\line(1,1){22}}
\put(1.9,11){\thicklines\line(1,1){22}}

\put(5,5){\thicklines\vector(1,3){4.5}}

\put(5,5){\thicklines\vector(2,3){17.5}}
\put(22.5,31.5){\line(-1,0){18.5}}
\put(22.5,31.5){\line(0,-1){27.5}}

\put(0,15){$\bm{a}_{1}$}
\put(7,22){$\bm{a}_{2}$}
\put(19,34){$\bm{d}$}

\put(12,1){$x_{1}$}
\put(0,24){$x_{2}$}

\end{picture}
\end{center}
\caption{A unique (upper left) and non-unique (upper right and bottom) solutions to the linear equation $A\bm{x}=\bm{d}$ in $\mathbb{R}_{\max,+}^{2}$.}\label{F-AxdS}
\end{figure}
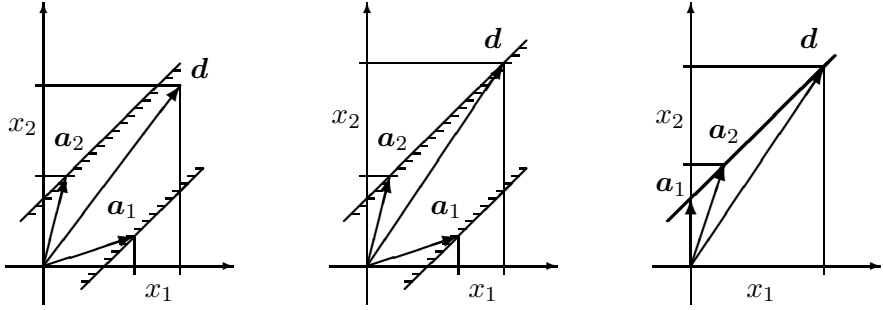

Two cases when equation \eqref{E-Axd} has more than one solution are shown on the Fig.~\ref{F-AxdS} (upper left and bottom).

In the upper left, the vector $\bm{d}$ is collinear with $\bm{a}_{1}$ and is not collinear with $\bm{a}_{2}$. The solution of the equation is any vector $\bm{x}$ with components
\begin{align*}
x_{1}
&=
(d_{1}^{-1}a_{11}\oplus d_{2}^{-1}a_{21})^{-1},
\\
x_{2}
&\leq
(d_{1}^{-1}a_{12}\oplus d_{2}^{-1}a_{22})^{-1}.
\end{align*}

In the case depicted in Fig.~\ref{F-AxdS} in the bottom, both vectors $\bm{a}_{1}$ and $\bm{a}_{2}$ are collinear with each other and with the vector $\bm{d}$. Under these conditions, there are two families of solution vectors $\bm{x}^{\prime}=(x_{1}^{\prime},x_{2}^{\prime})^{T}$ and $\bm{x}^{\prime\prime}=(x_{1}^{\prime\prime},x_{2}^{\prime\prime})$.

The vectors in the first family are given by
\begin{align*}
x_{1}^{\prime}
&=
(d_{1}^{-1}a_{11}\oplus d_{2}^{-1}a_{21})^{-1},
\\
x_{2}^{\prime}
&\leq
(d_{1}^{-1}a_{12}\oplus d_{2}^{-1}a_{22})^{-1},
\end{align*}
whereas those in the second family are given by
\begin{align*}
x_{1}^{\prime\prime}
&\leq
(d_{1}^{-1}a_{11}\oplus d_{2}^{-1}a_{21})^{-1},
\\
x_{2}^{\prime\prime}
&=
(d_{1}^{-1}a_{12}\oplus d_{2}^{-1}a_{22})^{-1}.
\end{align*}

\subsection{Systems of Equations and Inequalities}

Given matrices $A$ and $C$, and vectors $\bm{d}$ and $\bm{b}$, consider a problem to find vectors $\bm{x}$ that simultaneously solves an equation and an inequality combined into one system
\begin{equation}
\begin{split}
A\bm{x}
&=
\bm{d},
\\
C\bm{x}
&\leq
\bm{b}.
\end{split}
\label{S-AxdCxb}
\end{equation}

To solve the problem, we denote by $I$ a subset of indices for those columns in the matrix $A$ that form a minimal generating system for the vector $\bm{b}$, and by $\mathcal{I}$ a set of all such subsets. Furthermore, we define
$$
\widetilde{\mathcal{I}}
=
\{I\in\mathcal{I}|\bm{b}^{-}\bm{c}_{i}\leq\bm{d}^{-}\bm{a}_{i}, i\in I\}\subset\mathcal{I},
$$
where $\bm{a}_{i}$ and $\bm{c}_{i}$ are columns $i$ in the matrices $A$ and $C$, and recall the notation $\Delta=\Delta_{A}(\bm{d})=\sqrt{(A(\bm{d}^{-}A)^{-})^{-}\bm{d}}$.

\begin{lemma}\label{L-AxdCxb}
For any column regular matrices $A$ and $C$, and nonzero vectors $\bm{d}$ and $\bm{b}$, system \eqref{S-AxdCxb} has solutions if and only if $\Delta=\mathbb{1}$ and $\widetilde{\mathcal{I}}\ne\emptyset$. The general solution to the system is a (possible empty) family of solutions $\{\bm{x}_{I}|I\in\widetilde{\mathcal{I}}\}$, where each solution $\bm{x}_{I}=(x_{i})$ is given by
\begin{align*}
x_{i}
&=
(\bm{d}^{-}\bm{a}_{i})^{-},
&
\text{if $i\in I$},
\\
x_{i}
&\leq
(\bm{d}^{-}\bm{a}_{i}\oplus\bm{b}^{-}\bm{c}_{i})^{-},
&
\text{if $i\not\in I$}.
\end{align*}
\end{lemma}
\begin{proof}
System \eqref{S-AxdCxb} is solvable if and only if there exists a solution to the equation in the system to satisfy the condition $\bm{x}\leq(\bm{b}^{-}C)^{-}$, which is equivalent to the inequality at \eqref{S-AxdCxb}.

According to Theorem~\ref{T-GS}, the solution to the equation is a family $\{\bm{x}_{I}| I\in\mathcal{I}\}$, where each member is given by \eqref{S-xdaixdai}. Consider a solution that corresponds to a subset $I\in\mathcal{I}$. The solution is the vectors $\bm{x}_{I}=(x_{i})$ with components $x_{i}=(\bm{d}^{-}\bm{a}_{i})^{-}$, if $i\in I$, and $x_{i}\leq(\bm{d}^{-}\bm{a}_{i})^{-}$, otherwise.

These vectors include solutions to the inequality at \eqref{S-AxdCxb} if and only if $(\bm{d}^{-}\bm{a}_{i})^{-}\leq(\bm{b}^{-}\bm{c}_{i})^{-}$ for all $i\in I$. By collecting all sets $I$ that provides this condition, we get the set $\widetilde{\mathcal{I}}$.

It remains to see that each set $I\in\widetilde{\mathcal{I}}$ gives a solution $\bm{x}_{I}=(x_{i})$, where $x_{i}=(\bm{d}^{-}\bm{a}_{i}\oplus\bm{b}^{-}\bm{c}_{i})=(\bm{d}^{-}\bm{a}_{i})^{-}$, if $i\in I$, and $x_{i}\leq(\bm{d}^{-}\bm{a}_{i}\oplus\bm{b}^{-}\bm{c}_{i})^{-}$, otherwise.
\end{proof}

\subsection{An Extended Equation}

Consider a problem to solve with respect to the unknown vector $\bm{x}$ an extended equation in the form
\begin{equation}
A\bm{x}\oplus\bm{b}
=
\bm{d},
\label{E-Axbd}
\end{equation}
where $A$ is given matrix, $\bm{d}$ and $\bm{b}$ are given vectors.

In what follows, we assume that $\bm{b}\leq\bm{d}$ since if it is not the case, then equation \ref{E-Axbd} obviously has no solutions.

We introduce two sets of indices $I_{1}=\{i|b_{i}<d_{i}\}$ and $I_{2}=\{i|b_{i}=d_{i}\}$. Let $A_{1}$ and $A_{2}$ be submatrices composed of the rows in $A$ with indices from $I_{1}$ and $I_{2}$, respectively. Similarly, we define subvectors $\bm{d}_{1}$ and $\bm{d}_{2}$ for the vector $\bm{d}$, and subvectors $\bm{b}_{1}$ and $\bm{b}_{2}$ for the vector $\bm{b}$.

Equation~\eqref{E-Axbd} is then equivalent to a system
\begin{align*}
A_{1}\bm{x}
&=
\bm{d}_{1},
\\
A_{2}\bm{x}
&\leq
\bm{b}_{2}.
\end{align*}

In the same way as above, we construct a set $\mathcal{I}_{1}$ to include all sets of indices of minimal systems of columns in $A_{1}$ that generate $\bm{d}_{1}$. Furthermore, we reduce the set $\mathcal{I}_{1}$ to a set $\widetilde{\mathcal{I}}_{1}$ consisting of those $I\in\mathcal{I}_{1}$ that provide common solutions to both equation and inequality. Finally, we introduce the notation $\Delta_{1}=\Delta_{A_{1}}(\bm{d}_{1})=(A(\bm{d}_{1}^{-}A)^{-})^{-}\bm{d}_{1}$.

\begin{lemma}\label{L-Axbd}
For any column regular matrix $A$, and nonzero vectors $\bm{d}$ and $\bm{b}$, equation \eqref{E-Axbd} has solutions if and only if $\Delta_{1}=\mathbb{1}$ and $\widetilde{\mathcal{I}}_{1}\ne\emptyset$. The general solution to the system is a (possible empty) family of solutions $\{\bm{x}_{I}|I\in\widetilde{\mathcal{I}}_{1}\}$, where each solution $\bm{x}_{I}=(x_{i})$ is given by
\begin{align*}
x_{i}
&=
(\bm{d}^{-}\bm{a}_{i})^{-},
\qquad
\text{if $i\in I$},
\\
x_{i}
&\leq
(\bm{d}^{-}\bm{a}_{i})^{-},
\qquad
\text{if $i\not\in I$}.
\end{align*}
\end{lemma}
\begin{proof}
By applying Lemma~\ref{L-AxdCxb}, we immediately get the existence conditions as well as a general solution in the form of family $\{\bm{x}_{I}|I\in\widetilde{\mathcal{I}}_{1}\}$ with $\bm{x}_{I}=(x_{i})$ given by
\begin{align*}
x_{i}
&=
(\bm{d}_{1}^{-}\bm{a}_{i}^{1}\oplus\bm{b}_{2}^{-}\bm{a}_{i}^{2})^{-},
&
\text{if $i\in I$},
\\
x_{i}
&\leq
(\bm{d}_{1}^{-}\bm{a}_{i}^{1}\oplus\bm{b}_{2}^{-}\bm{a}_{i}^{2})^{-},
&
\text{if $i\not\in I$},
\end{align*}
where $\bm{a}_{i}^{1}$ and $\bm{a}_{i}^{2}$ denote columns $i$ in $A_{1}$ and $A_{2}$.

Since $\bm{b}_{2}=\bm{d}_{2}$, we have $\bm{d}_{1}^{-}\bm{a}_{i}^{1}\oplus\bm{b}_{2}^{-}\bm{a}_{i}^{2}=\bm{d}^{-}\bm{a}_{i}$ which leads to the solution in the desired form.
\end{proof}

An illustration of solutions to equation \eqref{E-Axbd} in $\mathbb{R}_{\max,+}^{2}$ are given in Fig.~\ref{F-Axbd}, where the set of vectors $\bm{y}=A\bm{x}\oplus\bm{b}$ for all $\bm{x}\in\mathbb{R}_{\max,+}^{2}$ is represented with a shaded part of the strip corresponding to the linear span of columns in $A$. In the case when the vector $\bm{b}$ is outside the strip (upper right and bottom), this set is extended by adding vertical segments drawn from the end point of $\bm{b}$ to the strip. 
\begin{figure}[ht]
\setlength{\unitlength}{1mm}
\begin{center}

\begin{picture}(33,45)

\put(0,5){\vector(1,0){33}}
\put(5,0){\vector(0,1){45}}

\put(5,5){\thicklines\vector(1,4){4.7}}

\put(5,5){\thicklines\vector(2,-1){5}}

\put(1,15){\line(1,1){25}}

\multiput(16.5,30.5)(1,1){10}{\line(1,0){1}}
\put(16,30){\thicklines\line(1,1){10}}

\put(7.5,0){\line(1,1){22}}

\multiput(23.5,16)(1,1){8}{\line(-1,0){1}}
\put(23.5,16){\thicklines\line(1,1){8}}

\put(5,5){\thicklines\vector(3,4){20}}

\put(5,5){\thicklines\vector(1,1){11}}

\put(16,16){\thicklines\line(0,1){14}}
\multiput(16,16)(0,1){15}{\line(1,0){1}}

\put(16,16){\thicklines\line(1,0){7.5}}
\multiput(16,16)(1,0){8}{\line(0,1){1}}

\put(13,0){$\bm{a}_{1}$}
\put(7,28){$\bm{a}_{2}$}

\put(26,33){$\bm{d}$}
\put(19,19){$\bm{b}$}

\end{picture}
\hspace{10\unitlength}
\begin{picture}(33,45)

\put(0,5){\vector(1,0){33}}
\put(5,0){\vector(0,1){45}}

\put(5,5){\thicklines\vector(1,4){4.7}}

\put(5,5){\thicklines\vector(2,-1){5}}

\put(1,15){\line(1,1){25}}

\multiput(23.5,37.5)(1,1){4}{\line(1,0){1}}
\put(23,37){\thicklines\line(1,1){4.5}}

\put(7.5,0){\line(1,1){22}}

\multiput(24,16.5)(1,1){7}{\line(-1,0){1}}
\put(23,15.5){\thicklines\line(1,1){8}}

\put(5,5){\thicklines\vector(3,4){23}}

\put(5,5){\thicklines\vector(4,1){18}}

\multiput(23,17.5)(0,1){20}{\line(1,0){1}}
\put(23,9.5){\thicklines\line(0,1){27.5}}

\put(13,0){$\bm{a}_{1}$}
\put(7,28){$\bm{a}_{2}$}

\put(30,37){$\bm{d}$}
\put(26,9){$\bm{b}$}

\end{picture}
\hspace{10\unitlength}
\begin{picture}(33,45)

\put(0,5){\vector(1,0){33}}
\put(5,0){\vector(0,1){45}}

\put(5,5){\thicklines\vector(1,4){4.7}}

\put(5,5){\thicklines\vector(2,-1){5}}

\put(1,15){\line(1,1){25}}

\multiput(23.5,37.5)(1,1){4}{\line(1,0){1}}
\put(23,37){\thicklines\line(1,1){4.5}}

\put(7.5,0){\line(1,1){25}}

\multiput(24,16.5)(1,1){8}{\line(-1,0){1}}
\put(23,15.5){\thicklines\line(1,1){8}}

\put(5,5){\thicklines\vector(2,1){18}}

\put(5,5){\thicklines\vector(4,1){18}}

\multiput(23,17.5)(0,1){20}{\line(1,0){1}}
\put(23,9.5){\thicklines\line(0,1){27.5}}

\put(13,0){$\bm{a}_{1}$}
\put(7,28){$\bm{a}_{2}$}

\put(26,14){$\bm{d}$}
\put(26,8){$\bm{b}$}

\end{picture}
\end{center}
\caption{Solutions to the equation $A\bm{x}\oplus\bm{b}=\bm{d}$ in $\mathbb{R}_{\max,+}^{2}$}\label{F-Axbd}
\end{figure}
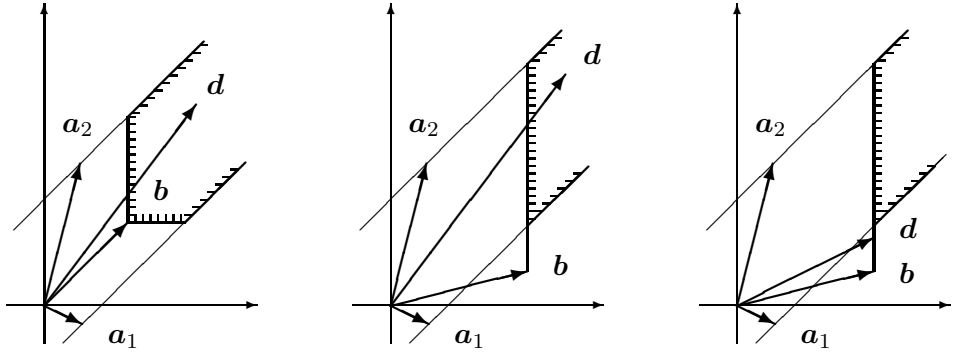

Note that a solution of equation \eqref{E-Axbd} may exist when equation \eqref{E-Axd} has no solutions as shown in Fig.~\ref{F-Axbd} (bottom).

\bibliographystyle{utphys}

\bibliography{Solution_of_linear_equations_and_inequalities_in_idempotent_vector_spaces}

\providecommand{\href}[2]{#2}\begingroup\raggedright\begin{thebibliography}{10}

\bibitem{Vorobjev1963Theextremal}
N.~N. Vorob{'}ev, ``The extremal matrix algebra,'' {\em Soviet Math. Dokl.}
  {\bfseries 4} no.~5, (1963) 1220--1223.

\bibitem{Vorobjev1967Extremal}
N.~N. Vorobjev, ``Algebra of positive matrices,'' {\em Elektronische
  Informationsverarbeitung und Kybernetik} {\bfseries 3} no.~1, (February,
  1967) 39--71. (in Russian).

\bibitem{Cuninghamegreen1979Minimax}
R.~Cuninghame-Green, {\em Minimax Algebra}, vol.~166 of {\em Lecture Notes in
  Economics and Mathematical Systems}.
\newblock Springer, Berlin, 1979.

\bibitem{Baccelli1993Synchronization}
F.~L. Baccelli, G.~Cohen, G.~J. Olsder, and J.-P. Quadrat, {\em Synchronization
  and Linearity: An Algebra for Discrete Event Systems}.
\newblock Wiley Series in Probability and Statistics. Wiley, Chichester, 1993.
\newblock \url{http://www-rocq.inria.fr/metalau/cohen/documents/BCOQ-book.pdf}.

\bibitem{Kolokoltsov1997Idempotent}
V.~N. Kolokoltsov and V.~P. Maslov, {\em Idempotent Analysis and Its
  Applications}, vol.~401 of {\em Mathematics and Its Applications}.
\newblock Kluwer Academic Publishers, Dordrecht, 1997.

\bibitem{Litvinov1998Correspondence}
G.~L. Litvinov and V.~P. Maslov,
  \href{http://dx.doi.org/10.1017/CBO9780511662508}{``The correspondence
  principle for idempotent calculus and some computer applications,''} in {\em
  Idempotency}, J.~Gunawardena, ed., Publications of the Newton Institute (No.
  11), pp.~420--443.
\newblock Cambridge University Press, Cambridge, 1998.
\newblock \href{http://arxiv.org/abs/0101021}{{\ttfamily arXiv:0101021
  [math.GM]}}.

\bibitem{Golan2003Semirings}
J.~S. Golan, {\em Semirings and Affine Equations Over Them: Theory and
  Applications}, vol.~556 of {\em Mathematics and Its Applications}.
\newblock Springer, New York, 2003.

\bibitem{Heidergott2006Maxplus}
B.~Heidergott, G.~J. Olsder, and J.~van~der Woude, {\em Max-plus at Work:
  Modeling and Analysis of Synchronized Systems}.
\newblock Princeton Series in Applied Mathematics. Princeton University Press,
  Princeton, 2006.

\bibitem{Butkovic2010Maxlinear}
P.~Butkovi\v{c}, \href{http://dx.doi.org/10.1007/978-1-84996-299-5}{{\em
  Max-linear Systems: Theory and Algorithms}}.
\newblock Springer Monographs in Mathematics. Springer, London, 2010.

\bibitem{Korbut1965Extremal}
A.~A. Korbut, ``Extremal spaces,'' {\em Soviet Math. Dokl.} {\bfseries 6}
  no.~5, (1965) 1358--1361.

\bibitem{Korbut1972Extremal}
A.~A. Korbut, ``Extremal vector spaces and their properties,'' {\em
  Elektronische Informationsverarbeitung und Kybernetik} {\bfseries 8} no.~8/9,
  (1972) 525--536. (in Russian).

\bibitem{Zimmermann1981Linear}
U.~Zimmermann, {\em Linear and Combinatorial Optimization in Ordered Algebraic
  Structures}, vol.~10 of {\em Annals of Discrete Mathematics}.
\newblock Elsevier, Amsterdam, 1981.

\bibitem{Olsder1988Cramer}
G.~Olsder and C.~Roos, ``Cramer and cayley-hamilton in the max algebra,'' {\em
  Linear Algebra Appl.} {\bfseries 101} no.~C, (1988) 87--108.

\bibitem{Cohen1989Algebraic}
G.~Cohen, P.~Moller, J.-P. Quadrat, and M.~Viot, ``Algebraic tools for the
  performance evaluation of discrete event systems,''
  \href{http://dx.doi.org/10.1109/5.21069}{{\em Proc. IEEE} {\bfseries 77}
  no.~1, (January, 1989) 39--85}.

\bibitem{Krivulin2005Onsolution}
N.~K. Krivulin, ``On solution of linear vector equations in idempotent
  algebra,'' in {\em Mathematical Models. Theory and Applications. Issue~5},
  M.~K. Chirkov, ed., pp.~105--113.
\newblock St.~Petersburg University, St.~Petersburg, 2005.
\newblock (in Russian).

\bibitem{Krivulin2009Onsolution}
N.~K. Krivulin, ``On solution of a class of linear vector equations in
  idempotent algebra,'' {\em Vestnik St.~Petersburg University. Ser.~10.
  Applied Mathematics, Informatics, Control Processes} no.~3, (2009) 64--77.
  (in Russian).

\bibitem{Krivulin2009Methods}
N.~K. Krivulin, {\em Methods of Idempotent Algebra for Problems in Modeling and
  Analysis of Complex Systems}.
\newblock St.~Petersburg University Press, St.~Petersburg, 2009.
\newblock \url{http://www.google.ru/books?id=PDQP7kIGrhMC}.
\newblock (in Russian).

\bibitem{Krivulin2012Asolution}
N.~Krivulin, ``A solution of a tropical linear vector equation,'' in {\em
  Advances in Computer Science}, S.~Yenuri, ed., vol.~5 of {\em Recent Advances
  in Computer Engineering Series}, pp.~244--249.
\newblock WSEAS Press, 2012.
\newblock \href{http://arxiv.org/abs/1212.6107}{{\ttfamily arXiv:1212.6107
  [math.OC]}}.

\end{thebibliography}\endgroup

\end{document}